\documentclass[12pt,draft]{amsart}
\usepackage{enumerate}
\usepackage{a4wide}
\usepackage{cancel}
\usepackage{float}
\usepackage{amsmath,amsthm,amssymb,latexsym,amsfonts}
\usepackage{multicol, multirow, longtable, mathtools}
\usepackage{paralist}

\newtheorem{theorem}{Theorem}[section]
\newtheorem{lemma}[theorem]{Lemma}
\newtheorem{proposition}[theorem]{Proposition}
\newtheorem{definition}[theorem]{Definition}
\newtheorem{corollary}[theorem]{Corollary}
\newtheorem{maintheorem}{Theorem}

\theoremstyle{remark}
\newtheorem{remark}[theorem]{Remark}
\theoremstyle{remark}

\newcommand{\C}{\ensuremath{\mathbb{C}}}
\newcommand{\R}{\ensuremath{\mathbb{R}}}

\newcommand{\g}[1]{\ensuremath{\mathfrak{#1}}}

\DeclareMathOperator{\ad}{ad}
\DeclareMathOperator{\Ad}{Ad}
\DeclareMathOperator{\Aut}{Aut}
\DeclareMathOperator{\codim}{codim}

\DeclareMathOperator{\ind}{ind}
\DeclareMathOperator{\Int}{Int}
\DeclareMathOperator{\Isom}{Isom}

\DeclareMathOperator{\rank}{rank}

\DeclareMathOperator{\tr}{tr}

\DeclareMathOperator{\RealPart}{Re}

\renewcommand{\Re}{\RealPart}

\newcommand\stru{\rule[-.65em]{0em}{1.7em}}
\newcommand\strt{\rule[-.15em]{0em}{1.2em}}

\begin{document}
	\title[Totally geodesic submanifolds in exceptional symmetric spaces]{Totally geodesic submanifolds in\\ exceptional symmetric spaces}
		\author[A. Kollross]{Andreas Kollross}
	\address{Institut f\"{u}r Geometrie und Topologie, Universit\"{a}t Stuttgart, Germany}
	\email{kollross@mathematik.uni-stuttgart.de}
	
	\author[A. Rodr\'{\i}guez-V\'{a}zquez]{Alberto Rodr\'{\i}guez-V\'{a}zquez}
	\address{KU Leuven, Department of Mathematics, Celestijnenlaan 200B, Leuven, Belgium}
	\email{alberto.rodriguezvazquez@kuleuven.be}
	\thanks{The second author has been supported by the projects:  PID2019-105138GB-C21/AEI/10.13039/501100011033 (Spain) and ED431C 2019/10, ED431F 2020/04 (Xunta de Galicia, Spain), the FPU program (Ministry of Education, Spain) and a fellowship within the Short-Term Grant Programme of the German Academic Exchange Service (DAAD),  and the Methusalem grant METH/15/026 of the Flemish Government (Belgium).}
	
\maketitle

\begin{abstract}
We classify maximal totally geodesic submanifolds in exceptional symmetric spaces up to isometry. Moreover, we introduce an invariant for certain totally geodesic embeddings of semisimple symmetric spaces, which we call the Dynkin index. We prove a result analogous to the index conjecture: for every irreducible symmetric space of non-compact type, there exists a totally geodesic submanifold which is of minimal codimension and whose non-flat irreducible factors have Dynkin index equal to one.
\end{abstract}

\section{Introduction}\label{sect:intro}
A totally geodesic submanifold $\Sigma$ of a Riemannian manifold $M$ is a connected submanifold of $M$ such that each geodesic of $\Sigma$ is also a geodesic of $M$.
The problem of classifying totally geodesic submanifolds in Riemannian symmetric spaces has been a relevant and outstanding topic of research in submanifold geometry in the last decades.
It was started in 1963 when, in his seminal paper, Wolf~\cite{wolfrankone} classified totally geodesic submanifolds in symmetric spaces of rank one. For rank two this problem has been addressed by Chen and Nagano~\cite{chen1,chen2} and Klein~\cite{kleindga,kleintams,kleinosaka}. Up to now, there are only complete classifications in symmetric spaces of rank less than three.

However, classification results for some special kinds of totally geodesic submanifolds are known. For example, totally geodesic submanifolds of maximal rank were studied by Ikawa and Tasaki~\cite{Ikawa}. They proved that a totally geodesic submanifold in a simple compact Lie group equipped with a bi-invariant metric is maximal if and only if it is a maximal subgroup or a Cartan embedding. Another special class of totally geodesic submanifolds is that of reflective submanifolds.  They are totally geodesic submanifolds which arise as connected components of the fixed point set of an involutive isometry. Reflective submanifolds have been classified by Leung in~\cite{Leung,LeungIII}. Moreover, Lagrangian totally geodesic submanifolds were classified in Hermitian symmetric spaces by Jaffee~\cite{jaffe1,jaffe2} and independently by~Leung~\cite{LeungIV}. Mashimo~\cite{Mashimo} classified totally geodesic surfaces in classical symmetric spaces.

A totally geodesic submanifold is said to be maximal if it is maximal with respect to the inclusion among totally geodesic proper  submanifolds. Note that every totally geodesic submanifold in a symmetric space can be extended to a complete totally geodesic submanifold.

In the present article, we classify maximal totally geodesic submanifolds in exceptional symmetric spaces. Since duality preserves totally geodesic submanifolds, it suffices to obtain the classification in the non-compact setting.
According to the classification by Cartan~\cite{helgason}, the irreducible symmetric spaces of non-compact type consist of several infinite families and the 17 exceptional spaces given in Table~\ref{table:ExcSp}.

\begin{table}[H]\caption{Exceptional symmetric spaces of non-compact type.}\label{table:ExcSp}
\begin{tabular}{l|lllll}

\stru $\mathsf{G}_2$-type & $\mathsf{G}^2_2/\mathsf{SO}_4$                       & $\mathsf{G}_2(\C)/\mathsf{G}_2$                  &                                                  &                                  &            \\ \cline{1-6}
\stru $\mathsf{F}_4$-type & $\mathsf{F}^4_4/\mathsf{Sp}_3\mathsf{Sp}_1$       & $\mathsf{F}^{-20}_4/\mathsf{Spin}_9$           & {$\mathsf{F}_4(\C)/\mathsf{F}_4$}                     &                                   &            \\ \hline
\stru $\mathsf{E}_6$-type & $\mathsf{E}^6_6/\mathsf{Sp}_4$                  & $\mathsf{E}^2_6/\mathsf{SU}_6\mathsf{Sp}_1$     & {$\mathsf{E}^{-14}_6/\mathsf{Spin}_{10}\mathsf{U}_1$} & {$\mathsf{E}^{-26}_6/\mathsf{F}_{4}$} & $\mathsf{E}_6(\C)/\mathsf{E}_{6}$ \\ \hline
\stru $\mathsf{E}_7$-type & $\mathsf{E}^{-5}_7/\mathsf{SO}_{12}\mathsf{Sp}_1$ & $\mathsf{E}^{-25}_7/\mathsf{E}_{6}\mathsf{U}_1$ & {$\mathsf{E}^{7}_7/\mathsf{SU}_{8}$}                  & {$\mathsf{E}_7(\C)/\mathsf{E}_{7}$}   &           \\ \cline{1-6}
\stru $\mathsf{E}_8$-type & $\mathsf{E}^{-24}_8/\mathsf{E}_{7}\mathsf{Sp}_1$  & $\mathsf{E}^{8}_8/\mathsf{SO}_{16}$             & {$\mathsf{E}_8(\C)/\mathsf{E}_{8}$}                   &                                  &           \\
\end{tabular}
\end{table}

\begin{maintheorem}\label{mainth:a}
Let $M=G/K$  be an irreducible exceptional symmetric space of non-compact type. Let $\Sigma$ be a maximal totally geodesic submanifold of~$M$. Then $\Sigma$ is isometric to one of the spaces listed in Tables~\ref{table:g2}, \ref{table:f4}, \ref{table:e6}, \ref{table:e7}, \ref{table:e8} at the end of this article. Conversely, every space listed in these tables can be isometrically embedded as a maximal totally geodesic submanifold of~$M$.
\end{maintheorem}

Since non-semisimple maximal totally geodesic submanifolds of irreducible symmetric spaces have been classified by Berndt and Olmos~\cite{BO1}, it suffices to study semisimple maximal totally geodesic submanifolds.
The proof of Theorem~\ref{mainth:a} uses a correspondence between maximal semisimple totally geodesic submanifolds and certain subalgebras of the Lie algebra of the isometry group. Note that a maximal semisimple totally geodesic submanifold needs not be maximal, since it might be contained in a non-semisimple maximal totally geodesic submanifold. While there is no classification of the relevant subalgebras, we construct a set which contains all of them. Then it remains to decide which of these subalgebras give rise to maximal totally geodesic submanifolds. In order to do so, we develop some criteria for maximality.

We would like to remark that our methods can be used to classify maximal totally geodesic submanifolds in symmetric spaces whose isometry group has rank less or equal than eight. Therefore, one could list all maximal totally geodesic submanifolds up to isometry in symmetric spaces with isometry group of rank less than or equal to eight. However, in this work we content ourselves with listing the classification in the exceptional symmetric spaces.

Onishchik~\cite{Onishchikindex} introduced an invariant of symmetric spaces concerning totally geodesic submanifolds called index, which is defined as the minimal codimension of a totally geodesic proper submanifold.
Berndt, Olmos and Rodr\'iguez~\cite{BO1,BO2,BO3,BO4,BO5} have computed the index~$i(M)$ of every irreducible symmetric space~$M$. In particular, they proved what they called the \emph{Index Conjecture}~\cite{BO1}. It states that in an irreducible symmetric space of non-compact type $M\neq \mathsf{G}^2_2/\mathsf{SO}_4$, there is some reflective submanifold~$\Sigma$ of~$M$ whose codimension equals the index of~$M$.

Generalizing a notion introduced by Dynkin~\cite{Dynkin}, we define the Dynkin index of certain semisimple subalgebras of simple real Lie algebras, see Definition~\ref{def:indD}. We use this to characterize the isometry types of totally geodesic embeddings of semisimple symmetric spaces into irreducible symmetric spaces.  This characterization allows us to derive a result analogous to the Index Conjecture:

\begin{maintheorem}
\label{th:index}
Let $M$ be an irreducible symmetric space of non-compact type. Then, there is some totally geodesic submanifold~$\Sigma$ in~$M$ with
$i(M)=\codim(\Sigma)$ such that the Dynkin index of the semisimple part of the Lie algebra of the isometry group of~$\Sigma$ equals~$(1,1,\dots,1)$.
\end{maintheorem}

This article is organized as follows. We revisit some well-known basic facts about symmetric spaces and their totally geodesic submanifolds in \S~\ref{subsect:pretot} and different formulations of the Karpelevich-Mostow theorem in \S~\ref{subsect:prekarpelevich}. Some useful facts about the complexification of Lie subalgebras of a real Lie algebra are recalled in \S~\ref{subsect:precomplex}.  In Section~\ref{sect:maximalsemi}, we prove a correspondence between maximal semisimple totally geodesic submanifolds in symmetric spaces of non-compact type and certain subalgebras of the isometry algebra of the ambient space. Then we specialize our study of totally geodesic submanifolds of symmetric spaces to the setting of exceptional symmetric spaces. In Section~\ref{sect:Dynkin}, we introduce an invariant for certain semisimple totally geodesic submanifolds in symmetric spaces of non-compact type that we call Dynkin index, which characterizes these submanifolds up to isometries. In Section~\ref{sect:totgeodreal}, we classify maximal totally geodesic submanifolds in exceptional symmetric spaces whose isometry group is absolutely simple, and in Section~\ref{sect:totgeodcomplex} we deal with the case when the isometry group is not absolutely simple. Section~\ref{sect:pfs} contains the proofs of the two main results (Theorem~\ref{mainth:a} and \ref{th:index} above).

\textbf{Acknowledgments}. The second author would like to thank his advisors Prof.~Jos\'e Carlos D\'iaz-Ramos and Prof.~Miguel Dom\'inguez-V\'azquez for their careful reading of the manuscript and their constant support. In particular, we want to thank Prof.~Miguel Dom\'inguez-V\'azquez for his helpful comments on~\S~2.2.

\section{Preliminaries}\label{sect:pre}

\subsection{Totally geodesic submanifolds in symmetric spaces}
\label{subsect:pretot}
In this section, we recall some basic facts and notations about Riemannian symmetric spaces and their totally geodesic submanifolds.

Let $M=G/K$ be a connected Riemannian symmetric space where $G=\mathrm{Isom}^0(M)$ is the connected component of the identity of the isometry group of~$M$ and the Lie group $K=\{g\in G: g\cdot o=o \}$ is the isotropy at some point $o\in M$. Let $\g{g}$ be the Lie algebra of~$G$.
Let $\mathcal{B}_{\g{g}}$ be the Killing form of~$\g{g}$, which is defined as $\mathcal{B}_{\g{g}}(X,Y)=\tr(\ad_X \ad_Y)$ for $X,Y\in \g{g}$, where $\ad$ stands for the adjoint representation of~$\g{g}$.

We consider the geodesic symmetry~$s_o$ at the base point $o\in M$; it gives rise to an involutive automorphism~$\sigma$ of~$G$ defined by $\sigma(g)= s_o g s_o$ whose differential at the identity~$\sigma_{*e}$ is denoted by~$\theta$. The map $\theta$ is a Lie algebra automorphism of~$\g{g}$, and $\g{g}$ decomposes as the direct sum of vector spaces $\g{g} = \g{k} \oplus \g{p}$, where $\g{k}$ is the fixed point set of~$\theta$ and $\g{p}$ is the eigenspace of~$\theta$ corresponding to the eigenvalue~$-1$.
In case $-\mathcal{B}_{\g{g}}(\theta X,Y)$ is positive definite, this splitting is called the \emph{Cartan decomposition} of~$\g{g}$ with respect to~$\theta$, and the involution~$\sigma$ is called a \emph{Cartan involution} of~$\g{g}$.

Note that for any real semisimple Lie algebra there exists a Cartan involution, and any two Cartan involutions in a real semisimple Lie algebra differ by an inner automorphism.

Furthermore, $\g{k}$ is the Lie algebra of~$K$ and we have the  bracket relations
$[\g{k},\g{k}]\subset \g{k},$ $[\g{k},\g{p}]\subset \g{p},$ $[\g{p},\g{p}]\subset \g{k}.$
We may identify $\g{p}$ with $T_o M$ by the linear isomorphism which maps $X\in \g{p}$ to $X^*_o=\left.\frac{d}{dt}\right|_{t=0} \left(\exp(t X)\cdot o\right)\in T_o M,$
where $\exp$ is the exponential map of the Lie group $G$.

The maximal abelian subspaces of~$\g{p}$ are conjugate by the isotropy group~$K$. In particular, they all have the same dimension, which is called the \emph{rank} of the symmetric space~$M$. It turns out that $\g{k}$ and~$\g{p}$ are orthogonal with respect to $\mathcal{B}_{\g{g}}$ and that every automorphism of~$\g{g}$ is a linear isometry for~$\mathcal{B}_{\g{g}}$.

We will say that a symmetric space is \emph{irreducible} if the universal cover $\widetilde{M}$ of~$M$, which is again a symmetric space, is not isometric to a non-trivial product of symmetric spaces.
A symmetric space is said to be of \emph{compact type}, \emph{non-compact type} or \emph{Euclidean type} if $\left.\mathcal{B}_{\g{g}}\right|_{\g{p}\times \g{p}}$, the restriction of the Killing form $\mathcal{B}_{\g{g}}$ to~$\g{p}$, is negative definite, positive definite or identically zero, respectively. When $M$ is of compact type,  $M$ is compact with non-negative sectional curvature and $G$ is a compact semisimple Lie group. If $M$ is of non-compact type, then $M$ is diffeomorphic to~$\mathbb{R}^n$, for some $n\ge2$, and $G$ is a non-compact semisimple Lie group.
The universal cover of a symmetric space splits as a product $$\widetilde{M}=M_0 \times M_{+} \times M_{-},$$ where $M_0$, which is called the \emph{flat factor}, is isometric to a Euclidean space and $M_{+}$ and $M_{-}$ are simply connected symmetric spaces of compact and non-compact type, respectively. It is said that $M$ is \emph{semisimple} if $M_0$ is a point.

An important notion, which establishes a relation between symmetric spaces of compact type and non-compact type, is duality. If we restrict our attention to simply connected symmetric spaces, there is a one-to-one correspondence between symmetric spaces of non-compact type and symmetric spaces of compact type. At the Lie algebra level this works as follows: Let $M=G/K$ be a symmetric space of non-compact type and let $\g{g}=\g{k}\oplus \g{p}$ be its Cartan decomposition. Consider $\g{g}_{\C}=\g{g}\otimes_{\R}\C$, the complexification of~$\g{g}$. We can define the subspace $\g{g}^*=\g{k}\oplus i\g{p}$ of~$\g{g}_{\C}$, where $i=\sqrt{-1}$. Then $\g{g}^*$ is a compact Lie algebra and $M^*=G^*/K^*$ is a symmetric space of compact type equipped with the Riemannian metric induced by the negative of the Killing form of~$\g{g}^*$, where $G^*$ is the simply connected Lie group with Lie algebra~$\g{g}^*$ and $K^*$ is the connected subgroup of~$G^*$ with Lie algebra~$\g{k}$.

Let $\Sigma$ be a connected totally geodesic submanifold of~$M=G/K$. By the homogeneity of~$M$, we can assume without loss of generality that $o\in \Sigma$. Cartan proved in~\cite{Cartan} that a totally geodesic submanifold~$\Sigma$ of~$M$ with $o\in \Sigma$ and $V=T_o \Sigma\subset T_o M$ exists if and only if $V\subset T_o M$ is \emph{curvature invariant}. This means that $R_o(V,V)V\subset V$, where $R$ is the Riemannian curvature tensor of~$M$. Using the identification of~$\g{p}$ and~$T_o M$, we can write the curvature tensor of~$M$ at~$o$ as
\[ R_o(X,Y)Z=[[X,Y],Z], \]
for $X,Y,Z \in T_o M$. Thus, a subspace $V\subset \g{p}$ is curvature invariant if and only if $[[X,Y],Z]\in V$ for every $X,Y,Z\in V$. A subspace~$V$ of~$\g{p}$ with this property is called a \emph{Lie triple system} in~$\g{p}$. Hence, there is a one-to-one correspondence between Lie triple systems $V$ in~$\g{p}$ and complete totally geodesic submanifolds $\Sigma$ in~$M$ containing~$o\in M$. In this article we will consider only complete totally geodesic submanifolds since every totally geodesic submanifold can be extended to a complete one. Furthermore, if $V$ is a Lie triple system in~$\g{p}$ such that its orthogonal complement in~$\g{p}$ is also a Lie triple system, we say that it is a \emph{reflective Lie triple system} and the corresponding totally geodesic submanifold is called \emph{reflective}. A submanifold is reflective if and only if it is a connected component of the fixed point set of an involutive isometry.

To obtain a homogeneous presentation of a totally geodesic submanifold from a Lie triple system we proceed as follows. Let $V\subset \g{p}$ be a Lie triple system in~$\g{p}$. Define $\g{g}':=[V,V]\oplus V\subset \g{k}\oplus \g{p}$, which is clearly a subalgebra of~$\g{g}$ since $V$ is a Lie triple system. It turns out that if we consider $G'$, the connected Lie subgroup of~$G$ with Lie algebra $\g{g}'$, then the $G'$-orbit through~$o\in M$ is a totally geodesic submanifold~$\Sigma\subset M$ with $T_o \Sigma = V$. This shows that if $\Sigma\subset M$ is a totally geodesic submanifold passing through~$o$, then $\g{p}_{\Sigma} := T_o\Sigma$ is a Lie triple system in~$\g{p}$ and we can define
\[ \g{k}_{\Sigma}:= [\g{p}_{\Sigma},\g{p}_{\Sigma} ], \qquad \g{g}_{\Sigma}:=\g{k}_{\Sigma} \oplus \g{p}_{\Sigma}.   \]
Then $\g{k}_{\Sigma}\subset \g{k}$ and $\g{g}_{\Sigma}\subset \g{g}$ are subalgebras and if we consider the connected Lie subgroups $G_{\Sigma}\subset G$ with Lie algebra~$\g{g}_{\Sigma}$ and $K_{\Sigma}\subset K$ with Lie algebra~$\g{k}_{\Sigma}$, then $\Sigma= G_{\Sigma}/K_{\Sigma}$ as homogeneous spaces.
Moreover, every totally geodesic submanifold~$\Sigma$ of~$M$ is invariant under~$s_p$ for every $p\in \Sigma$. Thus every totally geodesic submanifold~$\Sigma\subset M$ is a symmetric space with respect to its induced Riemannian metric. A totally geodesic submanifold is said to be \textit{semisimple} if it is a semisimple symmetric space. Finally, observe that if $V$ is a Lie triple system in~$\g{p}$, then $i V$ is a Lie triple system in~$i\g{p}$. This means that a totally geodesic submanifold~$\Sigma$ of~$G/K$ containing~$o$ corresponds to a totally geodesic submanifold~$\Sigma^*$ of~$G^*/K^*$ dual to~$\Sigma$. Hence, when studying totally geodesic submanifolds, it will not be restrictive to assume that our ambient symmetric space is either of compact type or of non-compact type.

An important invariant of a symmetric space is the index. The \emph{index} of a symmetric space~$M$, denoted by~$i(M)$, is the minimal codimension of a proper totally geodesic submanifold of~$M$. We will say that a totally geodesic submanifold~$\Sigma$ of~$M$ \emph{realizes the index} of~$M$ if $\Sigma$ has codimension equal to~$i(M)$. In a series of papers, Berndt, Olmos and Rodríguez~\cite{BO1,BO2,BO3,BO4,BO5} computed  the index of irreducible symmetric spaces.

\subsection{Karpelevich-Mostow Theorem}
\label{subsect:prekarpelevich}
A fundamental result in the study of totally geodesic submanifolds in symmetric spaces of non-compact type is known as the Karpelevich Theorem~\cite{karpelevich}, see also~\cite{mostow} and~\cite{discala-olmos}.
\begin{theorem}
	\label{th:karpelevichgeo}
	Let $M=G/K$ be a symmetric space of non-compact type. Then any connected semisimple subgroup $H \subset G$ acts on~$M$ with a totally geodesic orbit.
\end{theorem}

An equivalent, more algebraic formulation, see~\cite[Corollary 1, p.~46]{Onisbook}, is the following.
\begin{theorem}
	\label{th:karpelevichalg}
	Let $f\colon \g{h} \rightarrow \g{g}$ be a homomorphism of real semisimple Lie algebras and let a Cartan decomposition $\g{h} = \g{k}' \oplus \g{p}'$ be given. Then there exists a Cartan decomposition $\g{g} = \g{k} \oplus \g{p}$ such that $f(\g{k}') \subset \g{k}$ and $f(\g{p}')\subset \g{p}$.
\end{theorem}

A subalgebra~$\g{h}$ of a real semisimple  Lie algebra~$\g{g}$ is called \emph{canonically embedded}  with respect to some Cartan decomposition~$\g{g} = \g{k} \oplus \g{p}$ if $\g{h} = (\g{h}\cap\g{k}) \oplus (\g{h}\cap\g{p})$. This is equivalent to $\g{h}$ being $\theta$-invariant, where $\theta$ is the Cartan involution associated with the Cartan decomposition ~$\g{g} = \g{k} \oplus \g{p}$.  Let us recall that a subalgebra of a complex semisimple Lie algebra~$\g{g}$ is said to be \textit{reductive} if it is a reductive Lie algebra and its center consists of semisimple elements, i.e.\ of  elements $X \in \g{g}$ for which the linear map $\ad_X$ is a diagonalizable endomorphism of the vector space~$\g{g}$. For a real semisimple Lie algebra $\g{g}$ we say that a subalgebra $\g{h}\subset\g{g}$ is reductive if its complexification is reductive.

An \emph{algebraic group} over $\mathbb{K}\in\{\R,\C\}$ is an affine algebraic variety~$G$ over~$\mathbb{K}$ endowed with a group structure for which the map $G \times G\rightarrow G$, $(x, y)\mapsto xy^{-1}$ is polynomial.  It turns out that an algebraic group over $\mathbb{K}\in\{\R,\C\}$ is a Lie group over~$\mathbb{K}\in\{ \R,\C\}$.
An \emph{algebraic subgroup} of an algebraic group~$G$ is a closed subgroup of~$G$ (in the Zariski topology). An algebraic subgroup is itself an algebraic group.
A Lie algebra $\g{g}$ is \emph{algebraic} if it is the Lie algebra of some irreducible algebraic subgroup~$G$ of~$\mathsf{GL}_n(\mathbb{K})$, with $\mathbb{K}\in\{\R,\C \}$. In particular, semisimple Lie algebras are algebraic, see~\cite[p.~138]{Onishvinberg0}.
Let $G$ be an algebraic group over $\mathbb{K}\in\{ \R,\C\}$ with Lie algebra $\g{g}$.
A Lie subalgebra $\g{h}$ of $\g{g}$ is called \emph{algebraic} if there exists an algebraic subgroup~$H$ of~$G$ with Lie algebra~$\g{h}$.
For any Lie subalgebra~$\g{h}$ of~$\g{g}$ there is a smallest algebraic subalgebra~$\g{h}^a$ of~$\g{g}$ containing~$\g{h}$, called the \textit{algebraic closure} of~$\g{h}$ in~$\g{g}$.
\begin{remark}
\label{rem:alg}
Let $\g{g}$ be a semisimple Lie algebra.  By \cite[Chapter~1, \S6.2, Theorem~6.2]{Onishvinberg0}, if $\g{h}$ is a subalgebra of~$\g{g}$, then $[\g{h},\g{h}]=[\g{h}^a,\g{h}^a]$. Let $\g{l}$ be a maximal proper Lie subalgebra of~$\g{g}$. Then $[\g{l}^a,\g{l}^a]=[\g{l},\g{l}]\neq \g{g}$, and by maximality $\g{l}^a=\g{g}$. Hence a maximal proper Lie subalgebra of a semisimple Lie algebra is algebraic.
\end{remark}

A more general version of Theorem~\ref{th:karpelevichalg} can be formulated as follows, see~\cite[Chapter 6, Theorem 3.6]{Onishvinberg}.

\begin{theorem}[Karpelevich-Mostow]
\label{th:karpelevichmostow}
An algebraic subalgebra of a real semisimple Lie algebra~$\g{g}$ is reductive if and only if it is canonically embedded in~$\g{g}$ with respect to some Cartan decomposition of~$\g{g}$.
\end{theorem}

\subsection{Complexification of subalgebras}
\label{subsect:precomplex}
Let $\g{g}$ be a real Lie algebra. Recall that its \emph{complexification} is defined by $\g{g}_\C :=\g{g} \otimes_\R \C$.
Conversely, the \emph{realification}~$\g{h}_\R$ of a complex Lie algebra~$\g{h}$ is defined as the real Lie algebra obtained from~$\g{h}$ by restricting the scalars to the reals.
Furthermore, if $\g{g}$ is a complex Lie algebra, then a subalgebra~$\g{g}_0$ of~$\g{g}_\R$ is called a \emph{real form} of~$\g{g}$  if $\g{g} = \g{g}_0 + i\g{g}_0$ and $\g{g}_0 \cap i\g{g}_0=0$, where $i$ is the imaginary unit.
We say that a Lie algebra is of \emph{non-compact type} if it is semisimple and all of its simple ideals are non-compact Lie algebras.
We say that a subalgebra  of a Lie algebra~$\g{g}$ is \emph{maximal of non-compact type} if it is maximal among all proper subalgebras of non-compact type of~$\g{g}$.

A real Lie algebra $\g{g}$ is called \emph{absolutely simple} if it is simple and its complexification $\g{g}_\C$ is a simple complex Lie algebra. In case $\g{g}$ is an absolutely simple real Lie algebra, $\g{g}$ is a real form of the simple complex Lie algebra~$\g{g}_\C$. In case $\g{g}$ is a simple, but not absolutely simple, real Lie algebra, $\g{g}$ is the realification of a simple complex Lie algebra.
\begin{remark}\label{rem:notation}
	In later sections of this article, we do sometimes not distinguish in our notation between a complex Lie algebra and its realification when it is clear from the context which Lie algebra structure we consider.
\end{remark}

\begin{remark}
Recall from the classification of Riemannian symmetric spaces that there are two classes of irreducible Riemannian symmetric spaces of non-compact type. If $G$ is the connected component of the isometry group of such a space~$M=G/K$, we distinguish between the following two cases:
\begin{itemize}
\item Symmetric spaces of~\emph{type~III}: $\g{g}$ is \emph{absolutely simple}.
\item Symmetric spaces of~\emph{type~IV}: $\g{g}$ is \emph{simple, but not absolutely simple}.
\end{itemize}
\end{remark}

\begin{lemma}\label{lemma:complexification}
Let $\g{g}$ be a simple real Lie algebra and let $\g{h}\subset \g{g}$ be a subalgebra. Then the following statements hold:
\begin{enumerate}[i)]
\item  If $\g{h}_\C \subset  \g{g}_ \C$ is a maximal reductive   subalgebra, then $\g{h}\subset\g{g}$ is a maximal reductive  subalgebra.
\item If $\g{h}_\C\subset  \g{g}_ \C$ is a maximal semisimple subalgebra, then $\g{h}\subset\g{g}$ is a maximal semisimple subalgebra.
\end{enumerate}
\end{lemma}

\begin{proof}

By~\cite[\S2, Proposition~2(i)]{Onisbook}, $\g{h}_\C$ is a semisimple Lie algebra if and only if $\g{h}$ is a semisimple Lie algebra.  Thus, $[\g{h},\g{h}]$ is semisimple if and only if $[\g{h},\g{h}]_{\C}$ is semisimple. Since $[\g{h},\g{h}]_{\C}\oplus Z(\g{h})_\C=[\g{h}_\C,\g{h}_\C]\oplus Z(\g{h}_\C)$, $\g{h}$ is a reductive Lie algebra if and only if $\g{h}_\C$ is a reductive Lie algebra. Moreover, $\mathrm{ad}(Z(\g{h}))$ consists of semisimple elements if and only if $\mathrm{ad}(Z(\g{h}_\C))$ does. Consequently, we have proved that $\g{h}\subset \g{g}$ is a reductive  subalgebra if and only if $\g{h}_\C\subset \g{g}_\C$ is a reductive  subalgebra.
	
Let us assume that $\g{h}$ is a maximal reductive (resp. semisimple) subalgebra of $\g{g}$, and that there is some reductive  (resp.\ semisimple) subalgebra $\g{l}\subsetneq\g{g}$ such that $\g{h}\subset \g{l}\subset \g{g}$. Then $\g{h}_\C \subset \g{l}_\C\subset \g{g}_\C$ and $\g{l}_\C$ is reductive  (resp.\ semisimple). Since $\g{h}_\C$ is a maximal reductive   (resp.\ maximal semisimple) subalgebra, we have that $\g{h}_\C=\g{l}_\C$.
Therefore, $\g{h}$ and $\g{l}$ have the same dimension and we conclude that $\g{h}=\g{l}$.
\end{proof}

A subalgebra $\g{h}\subset \g{g}$ of a complex semisimple Lie algebra~$\g{g}$ is a \emph{regular} subalgebra if $\g{h}$ is normalized by some Cartan subalgebra $\g{a}\subset \g{g}$.  Following Dynkin~\cite{Dynkin}, we say that a subalgebra $\g{h}\subset \g{g}$ is an \emph{R-subalgebra} if it is contained in a regular proper subalgebra, and we say that it is an \emph{S-subalgebra} if it is not contained in a regular proper subalgebra.
In case where $\g{g}$ is not absolutely simple, its maximal subalgebras of non-compact type are described by Lemma~\ref{lemma:maximalcomplexsemisimple}, which is similar to~\cite[Appendix to Chapter~1, Theorem 1.6]{Dynkin2}.

\begin{lemma} 	\label{lemma:maximalcomplexsemisimple} 	
	Let $\g{g}$ be a simple complex Lie algebra and $\g{h} \subset \g{g}_\R$ be a subalgebra which is maximal among the subalgebras of non-compact type. Then $\g{h}$ coincides with one of the following: 	
	
	\begin{enumerate}[i)]
		\item a maximal semisimple regular subalgebra of~$\g{g}$,
		\item a maximal S-subalgebra of~$\g{g}$,	
		\item a non-compact real form of~$\g{g}$. 	
	\end{enumerate}
	Conversely, all of the above are maximal subalgebras of non-compact type of~$\g{g}_\R$.
\end{lemma}

\begin{proof}
	Let $\g{h}$ be a subalgebra which is maximal among the subalgebras of non-compact type of~$\g{g}_\R$.
	For all subalgebras of~$\g{g}_\R$ we have that
	$\g{h}+i\g{h}$ is a complex subalgebra of~$\g{g}$ and $\g{h}_0 := \{ X \in \g{h} \colon \lambda X \in \g{h} \hbox{~for~all~}\lambda \in \C \} = \g{h}\cap i\g{h}$ is an ideal of~$\g{h}+i\g{h}$, see~\cite[Appendix to Chapter~1]{Dynkin2}.
	If $\g{h}+i\g{h}=\g{g}$, then it follows that $\g{h}_0$ is an ideal of~$\g{g}$. Since $\g{g}$ is simple, it follows that~$\g{h}_0=0$ and hence that $\g{h}$ is a real form of~$\g{g}$. Obviously, $\g{h}$ is not a compact real form.
	
	Now we may assume that $\g{h}+i\g{h} \neq \g{g}$.
	Then $\g{h}+i\g{h}$ is contained in some maximal subalgebra~$\widetilde{\g{h}}$ of~$\g{g}$.
	Consider first the case where there is a maximal regular  reductive subalgebra~$\widetilde{\g{h}}$ containing $\g{h}+i\g{h}$. The maximal regular reductive subalgebras of simple complex Lie algebras were classified in~\cite[Chapter~II, \S5, No.~18]{Dynkin}.
	They are given in~\cite[Tables~12 and 12a]{Dynkin}, see also the remark at the beginning of the proof~\cite[Theorem~5.5]{Dynkin}.
	If the subalgebra $\widetilde{\g{h}}$ of the complex Lie algebra~$\g{g}$ is semisimple, then it is automatically a subalgebra of non-compact type of~$\g{g}_\R$. Hence in this case we have $\g{h} = \widetilde{\g{h}}$ and $\g{h}$ is one of the subalgebras given in~\cite[Table~12]{Dynkin}.
	On the other hand, if the subalgebra $\widetilde{\g{h}}$ of~$\g{g}$ is non-semisimple, then $\g{h}+i\g{h}$ is contained in one of the subalgebras in~\cite[Table~12a]{Dynkin} and since these are maximal semisimple regular subalgebras and also subalgebras of non-compact type in~$\g{g}_\R$, by maximality, $\g{h}$ is conjugate to one of them.
	
	It remains the case when $\g{h}+i\g{h}$ is not contained in a regular subalgebra of~$\g{g}$. Then it will be contained in a maximal S-subalgebra of~$\g{g}$. It follows from~\cite[Theorem~7.3]{Dynkin}, that every non-semisimple subalgebra of a complex semisimple Lie algebra is an R-subalgebra. Therefore we know that S-subalgebras of~$\g{g}$ are semisimple and hence are also subalgebras of non-compact type in~$\g{g}_\R$. By maximality, it follows now that $\g{h}$ is a maximal S-subalgebra of~$\g{g}$ in this case.
	
Let us now show that the converse holds. Since the real forms of~$\g{g}$ are simple they are either compact or of non-compact type. The subalgebras in \textit{i}) and, by \cite[No.~24, Theorem 7.3]{Dynkin}, \textit{ii}) are maximal semisimple subalgebras of~$\g{g}$ and they are of non-compact type, hence maximal subalgebras  of non-compact type.
\end{proof}

\section{Maximal semisimple totally geodesic submanifolds}
\label{sect:maximalsemi}
We prove below a theorem that establishes a one-to-one correspondence between maximal semisimple totally geodesic submanifolds in symmetric spaces of non-compact type and certain subalgebras of the Lie algebra of the isometry group of the ambient space.

A result of Alekseevsky and Di Scala~\cite[Proposition~ 5.5]{Alekseevsky} states that in a symmetric space of non-compact type $M= G/ K$, if the action of a subgroup of $G$ has two totally geodesic orbits, they are isomorphic as homogeneous spaces. Improving on this result, we prove the following statement which provides uniqueness up to congruence for the existence statement in Theorem~\ref{th:karpelevichgeo}. However, we note that, in this result, the group acting upon is not assumed to be semisimple.

\begin{proposition}
	\label{prop:congruency}
	Let $M=G/K$ be a symmetric space of non-compact type and $H\subset G$ be a connected Lie subgroup. Then all totally geodesic $H$-orbits are congruent in~$M$.
\end{proposition}
\begin{proof}
	Let $H\cdot p$ and $H\cdot q$ be two totally geodesic orbits of the action of~$H$ on $M$. By~\cite[Proposition~5.5]{Alekseevsky}, we have that $H\cdot p$ and $H\cdot q$ have the same dimension.
	
	Let us suppose that $H\cdot p$ is a point. Then, by the homogeneity of~$M$, we have that $H\cdot p$ and $H\cdot q$ are congruent.

	Now, let us assume that $\dim(H\cdot p)>0$. By~\cite[Proposition~5.5]{Alekseevsky} there exists an $H$-invariant totally geodesic submanifold~$N$ of~$M$ isometric to a Riemannian product $N = H \cdot p \times \R$ such that $H\cdot p, H\cdot q\subset N$. Moreover, let $\gamma$ be the geodesic starting at~$p\in H\cdot p$ whose initial velocity $\dot{\gamma}(0)$ lies in the orthogonal complement of $T_p (H\cdot p)$ in $T_p N$.  This is also a geodesic in~$M$ since $N$ is totally geodesic. Furthermore, $\gamma$ hits $H\cdot q$ at some point $q'\in M$, since $\exp\colon \nu(H\cdot p)\rightarrow M$ is a diffeomorphism, see~\cite[Proposition 3.5 \textit{i)}]{Duality}. We may assume $q'=\gamma(1)$. Let $x = \gamma(\frac12) \in N$. Then $s_x$, the geodesic reflection of~$M$ at~$x$, maps $p\in N$ to~$q'\in H\cdot q$ and preserves~$N$ since~$N$ is a totally geodesic submanifold in~$M$ containing $x$. Also, its differential sends $v\in T_p(H\cdot p)$ to $s_{x*}v\in T_{q'} N$ and
	\[0=\langle v,\dot{\gamma}(0)\rangle=\langle s_{x*} v, s_{x*} \dot{\gamma}(0)\rangle=-\langle s_{x*} v, \dot{\gamma}(1)\rangle.    \]
	This implies that $s_x(H\cdot p)$ is a totally geodesic submanifold of~$N$ passing through $q'\in N$ and orthogonal to~$\gamma$.
	The tangent space to any orbit of~$H$ is generated by Killing vector fields induced by $H$. Let $X$ be a Killing vector field induced by the action of~$H$. Then $\nabla X$ is skew-symmetric, where $\nabla$ stands for the Levi-Civita connection of~$N$. Thus $\langle \nabla_{\dot{\gamma}} X,\dot{\gamma}\rangle=0$, which implies that $\langle X,\dot{\gamma}\rangle $ is constant. Since $\gamma$ is  orthogonal to~$H\cdot p$ at~$p$,  it follows that $\gamma$ is also orthogonal to  $H\cdot q$ at~$q'$. Now we have shown that $s_x(H\cdot p)$ and $H\cdot q$, which both have codimension one in~$N$, are totally geodesic submanifolds of~$N$ passing through~$q'$ with the same tangent space. This shows that $H\cdot p$ and $H\cdot q$ are congruent in~$M$ via $s_x$.\qedhere	
\end{proof}
\begin{remark}
	Notice that the above result only works for symmetric spaces of non-compact type. For instance, the standard action of~$\mathsf{SO}_2$ on~$\mathrm{S}^2$ has two different isometry classes of totally geodesic orbits, namely, the equator and the poles.
\end{remark}

Let $M=G/K$ be an irreducible symmetric space of non-compact type, where $G$ is the connected component of the identity of the isometry group of~$M$ and $K$ is the isotropy subgroup~$G_o$ at~$o\in M$. Let $\Sigma$ be a (complete) totally geodesic submanifold of $M$ passing through $o\in M$.

As it can be deduced from the discussion in~\cite[\S 2]{BO3}, if $\Sigma$ is semisimple, then the Lie algebra $\g{g}_{\Sigma}:=[\g{p}_{\Sigma},\g{p}_{\Sigma}]\oplus \g{p}_{\Sigma}$, where $\g{p}_{\Sigma}=T_o \Sigma$, is isomorphic to the Lie algebra of the isometry group of~$\Sigma$, and $\g{g}_\Sigma$ is a semisimple Lie algebra in this case.
Recall that we say that a semisimple Lie algebra is of non-compact type if each simple ideal of it is non-compact.
If $\g{h}$ is a Lie algebra of non-compact type, then there is some symmetric space~$\Sigma_{\g{h}}$ of non-compact type such that the Lie algebra of its isometry group is $\g{h}$. In particular, if we consider a Cartan decomposition $\g{h}=\g{k}_{\g{h}} \oplus \g{p}_{\g{h}}$, we have $\g{k}_{\g{h}}=[\g{p}_{\g{h}},\g{p}_{\g{h}}]$.

Our ultimate aim would be to
to classify maximal totally geodesic submanifolds in~$M$.
Note that Berndt and Olmos classified in~\cite{BO1} the maximal totally geodesic submanifolds of~$M$ that are non-semisimple. So it remains to find those maximal totally geodesic submanifolds of~$M$ that are semisimple.

Our approach will consist in classifying first maximal semisimple totally geodesic submanifolds (i.e.\ the totally geodesic submanifolds that are maximal among the semisimple ones) and then discarding those that are contained in a non-semisimple totally geodesic submanifold. As announced, we will be able to carry out these tasks for the exceptional symmetric spaces, although many of the results in this article hold in more generality.

We will now prove a theorem which establishes a correspondence between maximal semisimple totally geodesic submanifolds of~$M$ and subalgebras which are maximal among subalgebras of non-compact type of~$\g{g}$.

\begin{theorem}[Correspondence Theorem]\label{th:correspondence}
	Let $M=G/K$ be an irreducible symmetric space of non-compact type and $\g{h} \subset \g{g}$ a subalgebra which is maximal among  subalgebras of non-compact type of~$\g{g}$. Then there is some $p\in M$ such that $\Sigma= H\cdot p$ is a maximal semisimple totally geodesic submanifold of~$M$, where $H\subset G$ is the connected subgroup of~$G$ with Lie algebra~ $\g{h}$.
	
	Conversely, if $\Sigma$ is a maximal semisimple totally geodesic submanifold of $M$, then there is a subalgebra $\g{g}_{\Sigma}$ of~$\g{g}$ which is maximal among subalgebras of non-compact type of~$\g{g}$ such that $G_{\Sigma}\cdot p=\Sigma$ for some $p\in M$, where $G_{\Sigma}$ is the connected subgroup of~$G$ with Lie algebra~ $\g{g}_{\Sigma}$.
\end{theorem}
\begin{proof}
	Let $\g{h}$ be a maximal subalgebra of  non-compact type of~$\g{g}$  and $H$ be the connected Lie subgroup of~$G$ with Lie algebra $\g{h}$. By Theorem~\ref{th:karpelevichalg}, we may assume that $\g{h}$ is canonically embedded with respect to the Cartan decomposition $\g{g}=\g{k}\oplus\g{p}$. Clearly, $\Sigma:=H\cdot o$ is then a semisimple totally geodesic submanifold of~$M$ and we claim that it is maximal among semisimple totally geodesic submanifolds. Let us assume that there is some semisimple totally geodesic submanifold~$\widetilde{\Sigma}$ such that $\Sigma\subset \widetilde{\Sigma}\subset M$.  We may assume $o\in \widetilde{\Sigma}$. Then $\widetilde{\Sigma}$ is of non-compact type. Let $\g{h}=\g{k}_{\g{h}}\oplus\g{p}_{\g{h}}$ be the Cartan decomposition of~$\g{h}$, where $\g{k}_{\g{h}}=\g{k}\cap\g{h}$ and $\g{p}_{\g{h}}=\g{p}\cap\g{h}$. It follows from our assumption that there is a Lie triple system $\g{p}_{\widetilde{\Sigma}}$ such that $\g{p}_{\g{h}}\subset \g{p}_{\widetilde{\Sigma}}\subset \g{p}$. Thus, $\g{k}_{\g{h}}=[\g{p}_{\g{h}},\g{p}_{\g{h}}]\subset [\g{p}_{\widetilde{\Sigma}},\g{p}_{\widetilde{\Sigma}}]=:\g{k}_{\widetilde{\Sigma}}$ and $\g{h}\subset \g{g}_{\widetilde{\Sigma}} := \g{k}_{\widetilde{\Sigma}} \oplus \g{p}_{\widetilde{\Sigma}}$. However,  $\g{g}_{\widetilde{\Sigma}}$ is of non-compact type, since $\widetilde{\Sigma}$ is of non-compact type, and $\g{h}$ is maximal among  subalgebras of~$\g{g}$ of non-compact type. It follows that $\g{h}=\g{g}_{\widetilde{\Sigma}}$ and $\widetilde{\Sigma}=\Sigma$.
	
	Let  $\Sigma$ be a maximal semisimple totally geodesic submanifold passing through $o\in M$ and let $\g{p}_\Sigma$ be the tangent space to~$\Sigma$ at~$o$.	
	Consider $\g{g}_{\Sigma} := [\g{p}_\Sigma,\g{p}_\Sigma] \oplus \g{p}_\Sigma$. Clearly, $\g{g}_{\Sigma}$ is of non-compact type since it is isomorphic to the Lie algebra of the isometry group of~$\Sigma$. We will prove that $\g{g}_{\Sigma}$ is maximal among subalgebras of non-compact type of~$\g{g}$. Let $\g{h}$ be a subalgebra of~$\g{g}$ of non-compact type such that
	$\g{g}_{{{\Sigma}}}\subset \g{h}\subset \g{g}$. We will prove that $\g{g}_{\Sigma}=\g{h}$.  We apply Theorem~\ref{th:karpelevichalg} to~$\g{g}_{{{\Sigma}}}\subset \g{h}$ to find a Cartan decomposition for $\g{h}$ such that
	\[  \g{h}=\g{k}_{\g{h}} \oplus \g{p}_{\g{h}} \qquad \mbox{satisfying} \qquad \g{k}_{{\Sigma}}:=[\g{p}_{\Sigma},\g{p}_{\Sigma}]\subset \g{k}_{\g{h}}, \qquad \g{p}_{{\Sigma}}\subset \g{p}_{\g{h}}.  \]
	Applying Theorem~\ref{th:karpelevichalg} once more to~$\g{h}\subset \g{g}$, we find a Cartan decomposition for $\g{g}$ such that
	\[  \g{g}=\g{k}'\oplus \g{p}' \qquad \mbox{satisfying} \qquad \g{k}_{{\Sigma}}\subset \g{k}_{\g{h}}\subset \g{k}', \qquad \g{p}_{{\Sigma}}\subset \g{p}_{\g{h}}\subset \g{p}'.  \]
	Now since any two Cartan involutions in a real semisimple Lie algebra differ by an inner automorphism, there is some $g\in G$ such that $\g{p}'=\Ad(g)\g{p}$. Thus $\Ad(g^{-1})\g{p}_{\Sigma}$ and $\Ad(g^{-1})\g{p}_{\g{h}}$ are Lie triple systems in~$\g{p}$ since $\Ad(g)\in \Aut(\g{g})$ and $\Ad(g^{-1})\g{p}_{\Sigma}\subset  \Ad(g^{-1})\g{p}_{\g{h}}\subset \g{p}$.  Let $H$ and $G_\Sigma$ be the Lie subgroups of $G$ with Lie algebras $\g{h}$ and $\g{g}_{\Sigma}$, respectively.   If we consider $p:=g\cdot o\in M$, we have that $G_{\Sigma}\cdot p$ and $H\cdot p$ are totally geodesic submanifolds in~$M$, since $g^{-1} G_{\Sigma} g \cdot o$ and  $g^{-1} H g \cdot o$ are totally geodesic, too.  Furthermore, $G_{{\Sigma}}\cdot p$ is contained in~$H\cdot p$. However, $\Sigma=G_{\Sigma}\cdot o$ is maximal among the semisimple totally geodesic submanifolds passing through $o\in M$, hence, by Proposition~\ref{prop:congruency}, $G_{\Sigma}\cdot p$ is maximal among the semisimple totally geodesic submanifolds passing through $p\in M$. Then $G_{{\Sigma}}\cdot p= H\cdot p$. Therefore, $\g{p}_{\g{h}}=\g{p}_{\Sigma}$ and consequently $\g{g}_{\Sigma}=\g{h}$, since  $\g{g}_{\Sigma}$ and $\g{h}$ are of non-compact type.
\end{proof}

\section{Dynkin index and totally geodesic submanifolds}\label{sect:Dynkin}
In this section, we extend the definition of the Dynkin index of a simple subalgebra of a simple complex Lie algebra to certain classes of semisimple subalgebras of simple real Lie algebras, and we use it to characterize isometry classes of totally geodesic submanifolds.

Recall that we denote by~$\mathcal{B}_{\g{g}}$ the Killing form of a Lie algebra~$\g{g}$.
Let $\g{g}$ be a simple complex Lie algebra and let $\g{a}$ be a Cartan subalgebra of~$\g{g}$. Let $\Delta$ be the set of roots with respect to~$\g{a}$. It is shown in~\cite[Chapter~III, Theorem~4.2]{helgason} that the restriction of $\mathcal{B}_{\g{g}}$ to~$\g{a}$ is non-degenerate, and that we hence may identify each root $\alpha \in \Delta$ with a vector $H_\alpha \in \g{g}$ such that $\alpha(H) = \mathcal{B}_{\g{g}}(H,H_\alpha)$ for all~$H \in \g{a}$.
It follows from~\cite[Chapter~III, Theorem~4.4(i)]{helgason} that the numbers $q_\alpha := \mathcal{B}_{\g{g}}( H_\alpha , H_\alpha)$ are positive for all~$\alpha \in \Delta$. Let $q := \max \{q_\alpha : \alpha \in \Delta\}$ and define the bilinear form $Q_{\g{g}}$ on~$\g{g}$ by
\[
Q_{\g{g}} := \frac2q \mathcal{B}_{\g{g}},
\]
i.e.\ $Q_{\g{g}}$ is the multiple of the Killing form normalized by the condition that the square of the length of the longest root equals~$2$.
Let us recall from~\cite[\S2]{Dynkin} Dynkin's definition of the index of a simple subalgebra of a simple complex Lie algebra. Let $\g{h}$ and $\g{g}$ be simple complex Lie algebras and let $f \colon \g{h} \to \g{g}$ be a Lie algebra monomorphism.
By Schur's Lemma, the number~$\ind_D(f)$, given by
\begin{equation}\label{eq:Dynkinindex}
\ind_D(f) \cdot Q_{\g{h}}(X,Y)
=Q_{\g{g}}(f(X),f(Y))\quad\hbox{for~all~$X,Y \in \g{h}$},
\end{equation}
is well defined. It is called the \emph{Dynkin index} of the subalgebra~$f(\g{h})$ in~$\g{g}$. Dynkin proved in~\cite[Theorem~2.2]{Dynkin} that it is a positive integer.
If $\g{h} \subset \g{g}$ is a subalgebra of the complex Lie algebra~$\g{g}$, we will write $\ind_D(\g{h,\g{g}})$ or, $\ind_D(\g{h})$, when the embedding is clear from the context, for the Dynkin index of the inclusion map.

Now let us mention the multiplicative property of the Dynkin index, see~\cite[\S2, No.~7]{Dynkin}. Observe that given the following embeddings of complex simple Lie algebras $\g{h}_1\subset \g{h}_{2}\subset\g{g}$, then
\[ \ind_D(\g{h}_1, \g{g})=\ind_D(\g{h}_1, \g{h}_2)\cdot\ind_D(\g{h}_2, \g{g}).  \]

We would like to extend the definition of the Dynkin index to (semi-)simple subalgebras of simple real Lie algebras.
Let $\g{h} = \g{h}_1 \oplus \dots \oplus \g{h}_n$ be a semisimple subalgebra of the simple complex Lie algebra~$\g{g}$, where  $\g{h}_s$ is a simple ideal for every $s\in\{1,\ldots,n\}$. We define \[\ind_D(\g{h}) := ( \ind_D(\g{h}_1), \dots, \ind_D(\g{h}_n)).\]

For an absolutely simple real Lie algebra~$\g{g}$, we will define the Dynkin index of a semisimple subalgebra~$\g{h}$ as the Dynkin index of~$\g{h}_\C$ in~$\g{g}_\C$. This is well defined since $\g{h}_\C$ is semisimple, see the proof of Lemma~\ref{lemma:complexification}.
In order to define the Dynkin index also for a semisimple subalgebra of a simple real Lie algebra which is not absolutely simple, i.e.\ whose complexification is not simple, we prove the following.

\begin{lemma}\label{lemma:semisplcpx}
Let $\g{g}$ be a simple complex Lie algebra and let $\g{h}$ be a semisimple subalgebra of the realification~$\g{g}_\R$ of~$\g{g}$.
Then $\g{h}+i\g{h}$ is a semisimple subalgebra of~$\g{g}$.
\end{lemma}

\begin{proof}
Notice that $\g{h}+i\g{h}$ is a subalgebra of~$\g{g}$ and $\g{h}\cap i\g{h}$ is an ideal of~$\g{h}+i\g{h}$, see~\cite[Appendix to Chapter~1, Lemma 1.1]{Dynkin2}.
Hence $\g{h}\cap i\g{h}$ is also an ideal of the subalgebra $\g{h}$ of~$\g{g}_\R$.
Let $\g{h} = \g{h}_1 \oplus \dots \oplus \g{h}_n$, where the~$\g{h}_j$ are the simple ideals of~$\g{h}$.
Since $\g{h}_j$ is simple, we have either $\g{h}_j\cap i\g{h}_j=\g{h}_j$, and then $\g{h}_j+i \g{h}_j=\g{h}_j$ is simple, or $\g{h}\cap i\g{h} = 0$, in which case $\g{h}_j$ is a simple real form of the complex Lie algebra $\g{h}_j+i\g{h}_j$.
Note also that $i\g{h}_j \cap \g{h}_k = 0$ for $j \neq k$. Indeed, we have $[i\g{h}_j \cap \g{h}_k,\g{h}_k] \subset i [\g{h}_j,\g{h}_k] = 0$ for $j \neq k$.
Thus we may assume, by renumbering the~$\g{h}_j$, if necessary, that there is a $k \in \{0,\dots,n\}$ such that multiplication by~$i$ maps $\g{h}_j$ to itself if and only if $j < k$. It follows that
\[
\g{h}+i\g{h} = \g{h}_1 \oplus \dots \oplus \g{h}_{k-1} \oplus  (\g{h}_{k}+i\g{h}_{k}) \oplus \dots \oplus (\g{h}_n+i\g{h}_n).
\]
Hence $\g{h}+i\g{h}$ is semisimple.
\end{proof}

\begin{definition}\label{def:indD}
Let $\g{g}$ be a simple real Lie algebra and let $\g{h}$ be a semisimple subalgebra of~$\g{g}$.
\begin{enumerate}[i)]
  \item If $\g{g}$ is absolutely simple, then we define
  \[
  \ind_D(\g{h}, \g{g}) := \ind_D(\g{h}_\C , \g{g}_\C).
  \]
  \item If $\g{g} = \g{l}_\R$ is the realification of a simple complex Lie algebra~$\g{l}$ and $\g{h}$ is a complex subalgebra  or a real form of $\g{l}$, then we define
  \[
  \ind_D(\g{h}, \g{g}) := \ind_D(\g{h}+i\g{h} , \g{l}).
  \]
\end{enumerate}
Furthermore, let $\g{h}$ and $\g{h}'$ be two semisimple subalgebras of~$\g{g}$ for which the Dynkin index is defined. We say that $\g{h}$ and $\g{h}'$ are \emph{isometric} if there is a Lie algebra isomorphism $\varphi \colon \g{h} \to \g{h}'$ such that each simple ideal of~$\g{h}$ is mapped onto a simple ideal of~$\g{h}'$ of the same Dynkin index.
\end{definition}

\begin{remark}\label{rem:indD}
	Note that, with this definition, for a simple subalgebra~$\g{h}$ of an absolutely simple real Lie algebra~$\g{g}$ the Dynkin index of the subalgebra~$\g{h}_\C$ of~$\g{g}_\C$ is either a natural number (in case $\g{h}$ is absolutely simple) or, by the following lemma, a pair of natural equal numbers (in case $\g{h}$ is not absolutely simple).
\end{remark}

\begin{lemma}\label{lemma:eqDI}
Let $\g{g}$ be an absolutely simple real Lie algebra and let $\g{h}$ be a simple, but not absolutely simple subalgebra of~$\g{g}$. Then $\g{h}_\C = \g{h}_1 \oplus \g{h}_2$, where $\g{h}_1$ and $\g{h}_2$ are two isomorphic simple subalgebras of~$\g{g}_\C$ of equal Dynkin index.
\end{lemma}

\begin{proof}
We may assume that $\g{h}$ is canonically embedded with respect to a Cartan decomposition $\g{g} = \g{k} \oplus \g{p}$ by Theorem~\ref{th:karpelevichalg}. Let $\theta$ be the corresponding Cartan involution.
Since $\g{h}$ is simple, but not absolutely simple, it is isomorphic to the realification of a simple complex Lie algebra and since $\g{h}$ is canonically embedded with respect to the above Cartan decomposition, it follows that $\g{h} \cap \g{k}$ is a compact real form of this simple complex Lie algebra. In particular, $\theta$, restricted to~$\g{h}$, is a non-trivial involutive automorphism of~$\g{h}$.

It is well known that the complexification of~$\g{f}_\R$, where $\g{f}$ is a complex Lie algebra, is isomorphic to~$\g{f} \oplus \bar{\g{f}}$, where $\bar{\g{f}}$ denotes the complex conjugate Lie algebra of $\g{f}$, see e.g.~\cite[\S2, Proposition~3]{Onisbook}. Note that the complex Lie algebras that have a real form are isomorphic to their complex conjugates via an antilinear map. This shows that $\g{h}_\C = \g{h}_1 \oplus \g{h}_2$, where $\g{h}_1$ and $\g{h}_2$ are isomorphic.

Let $\tau \colon \g{g}_\C \to \g{g}_\C$ be the map defined by $\tau(X+iY)=X-iY$ for $X,Y \in \g{g}$. It is straightforward to check that $\tau$ is an automorphism of~$\g{g}_\R$.
The map $\tau$ obviously leaves~$\g{h}_\C$ invariant and it acts on~$\g{h}_\C$ as an automorphism of the real Lie algebra~$(\g{h}_\C)_\R$. Since $\g{h}_1$ and~$\g{h}_2$ are isomorphic simple ideals of~$(\g{h}_\C)_\R$, we have either $\tau(\g{h}_1)=\g{h}_{1}$ or $\tau(\g{h}_1)=\g{h}_{2}$. Assume we are in the former case. Then we also have $\tau(\g{h}_2)=\g{h}_{2}$. The fixed point set of the involution~$\tau$ on~$\g{h}_\C$ is the direct sum~$\g{k}_1 \oplus \g{k}_2$, where $\g{k}_j$ is a proper subalgebra of~$\g{h}_j$ for $j=1,2$.
However, the fixed point set of the action of~$\tau$ on~$\g{h}_\C$ coincides with the simple Lie algebra~$\g{h}$ and we have arrived at a contradiction.

We have shown that $\tau(\g{h}_1)=\g{h}_{2}$.
Since $\g{h}_1$ is simple, there are vectors $X,Y \in \g{h}_1$ such that $\mathcal{B}_{\g{h}_1}(X,Y)\neq0$ and we have by~\cite[\S~2, Proposition~2(ii)]{Onisbook}
\[
\ind_D(\g{h}_1,\g{g})
= \frac{Q_{\g{g}}(X,Y)}{Q_{\g{h}_1}(X,Y)}
= \frac{\overline{Q_{\g{g}}(\tau(X),\tau(Y))}}{\;\;\overline{Q_{\g{h}_2}(\tau(X),\tau(Y))}\;\;}
= \overline{\ind_D(\g{h}_2,\g{g})}.
\]
Since $\ind_D(\g{h}_2,\g{g})$ is a natural number, it follows that the subalgebras $\g{h}_1$ and $\g{h}_2$ have the same Dynkin index.
\end{proof}

Hence, we have defined the Dynkin index for every semisimple subalgebra~$\g{h}$ of an absolutely simple real Lie algebra~$\g{g}$. Also, we have defined the Dynkin index for every semisimple complex subalgebra $\g{h}$ of the realification of a complex simple Lie algebra $\g{g}$ and for its real forms. In both cases, we denote it by $\ind_D(\g{h},\g{g})$ or, $\ind_D(\g{h})$, when the embedding is clear from the context.
It immediately follows from this definition that the Dynkin index is one if $\g{h}$ is a real form of the simple complex Lie algebra~$\g{g}$. Indeed, the following result shows that isometric real forms of $\g{g}$ induce congruent, and hence isometric, totally geodesic submanifolds of $M$.
\begin{lemma}
	\label{lemma:congrealforms}
	Let $M=G/K$ be a symmetric space, where $G$ is a simple complex Lie group. Let $\g{h}$ and $\g{h}'$ be isomorphic real forms of a complex simple Lie algebra $\g{g}$. Then, the totally geodesic orbits of $H$ and $H'$ in $M$ are all congruent in $M$, where $H$ and $H'$ are the connected subgroups of $G$ with Lie algebras $\g{h}$ and $\g{h}'$, respectively.
\end{lemma}
\begin{proof}
	Let $\g{h}$ and $\g{h'}$ be isomorphic real forms in $\g{g}$.  Then, there is some Lie algebra isomorphism $f\colon \g{h}\rightarrow\g{h'}$.  By complexifying, this map extends to a Lie algebra automorphism $\widetilde{f}$ of $\g{g}$.
	Let us fix some Cartan decomposition $\g{g}=\g{k}\oplus\g{p}$. We can assume without loss of generality that $H\cdot o$ is a totally geodesic submanifold in $M$. This implies that there is some Cartan decomposition $\g{h}=\g{p}_{\g{h}}\oplus\g{k}_{\g{h}}$ such that $\g{p}_{\g{h}}\subset\g{p}$ and $\g{k}_{\g{h}}\subset\g{k}$. By Theorem \ref{th:karpelevichalg}, there is some Cartan decomposition $\g{p}_{\g{h}'}\oplus\g{k}_{\g{h'}}=\g{h'}$ such that $\widetilde{f}(\g{p}_{\g{h}})=\g{p}_{\g{h'}}$ and $\widetilde{f}(\g{k}_{\g{h}})=\g{k}_{\g{h'}}$.
	
	Furthermore, two Cartan decompositions of $\g{g}$ are conjugate in $\Int(\g{g})$. Hence there is some $g\in G$ such that $\varphi:=\Ad(g)\circ\widetilde{f}\in\Aut(\g{g})$, $\varphi(\g{p})=\g{p}$ and $\varphi(\g{h})\subset\g{g}$ is a real form of $\g{g}$ conjugate to $\g{h'}$ in $\g{g}$.
	Furthermore, $\varphi$ preserves the curvature tensor of $M$ at $o$ since this is given by Lie brackets. Hence, $\varphi$ is a linear isometry of $\g{p}$ which preserves sectional curvature at $o$. Thus, by \cite[Corollary 2.3.14]{wolf}, $\varphi$ extends to an isometry $k\in \Isom(M)$ which fixes $o\in M$, since it leaves $\g{p}$ invariant. Thus, we have that $k(H\cdot o)=g H' g^{-1}\cdot q$, for certain $q\in M$, which implies that there is a totally geodesic orbit of $H$ which is congruent to a totally geodesic orbit of $H'$. Consequently, by Proposition \ref{prop:congruency}, the totally geodesic orbits of $H$ and $H'$ are all congruent in $M$.
\end{proof}

We have defined the Dynkin indices of certain semisimple subalgebras of simple real Lie algebras in such a way that isometry classes of subalgebras of the isometry algebra of an irreducible symmetric space of non-compact type correspond to isometry classes of totally geodesic submanifolds. This is the content of the following theorem.

\begin{theorem}\label{th:isom}
Let $M=G/K$ be an irreducible symmetric space of non-compact type. Let $\Sigma_1,\Sigma_2$ be two semisimple totally geodesic submanifolds containing $o \in M$. For $j=1,2$ let $\g{p}_j = T_o\Sigma_j$ and let $\g{g}_{\Sigma_j} = [\g{p}_j,\g{p}_j] \oplus \g{p}_j\subset\g{g}$ and  assume that one of the following holds:
\begin{enumerate}[i)]
	\item $\g{g}$ is absolutely simple.
	\item $\g{g}$ is complex and $\g{g}_{\Sigma_j}$ is  a complex subalgebra or a real form of $\g{g}$ for each $j\in\{1,2\}$.
\end{enumerate}
Then $\Sigma_1$ and $\Sigma_2$ are isometric  if and only if $\g{g}_{\Sigma_1}$ and $\g{g}_{\Sigma_2}$ are isometric subalgebras of~$\g{g}$.
\end{theorem}

\begin{proof}
It suffices to prove the statement in the case when the $\g{g}_{\Sigma_j}$, $j=1,2$, are simple.
Let us consider the Cartan decomposition~$\g{g} = \g{k} \oplus \g{p}$.
Since any two Cartan involutions of~$\g{g}$ differ by an inner automorphism of $\g{g}$, we may assume by  Theorem~\ref{th:karpelevichalg} that $\g{g}_{\Sigma_1}$ and $\g{g}_{\Sigma_2}$ are both canonically embedded into~$\g{g}$ with respect to the  Cartan decomposition~$\g{g} = \g{k} \oplus \g{p}$.
Then we have $\Sigma_j = \exp_o(\g{g}_{\Sigma_j}\cap \g{p})$, where $\exp_o$ denotes the Riemannian exponential map of~$M$ at the point~$o$.

If we assume that $\Sigma_1$ and $\Sigma_2$ are isometric, it follows that there is an isomorphism
$\varphi \colon \g{g}_{\Sigma_1} \to \g{g}_{\Sigma_2}$. Conversely, if $\g{g}_{\Sigma_1}$ and $\g{g}_{\Sigma_2}$ are isometric subalgebras of~$\g{g}$, they are isomorphic by definition.

Since $M$, $\Sigma_1$ and $\Sigma_2$ are irreducible symmetric spaces, their invariant Riemannian metrics are unique up to scaling by some constant factor. Therefore, we may assume that $M$ is endowed with the $G$-invariant metric induced by the Killing form of~$\g{g}$ and we know that the invariant Riemannian metrics induced on~$\Sigma_j$ are obtained from the metrics induced by the Killing form of~$\g{g}_{\Sigma_j}$  by applying a constant scaling factor~$c_j$ to this metric.
Thus it remains to show that $c_1=c_2$ if and only if the Dynkin indices of the subalgebras $\g{g}_{\Sigma_1}$ and $\g{g}_{\Sigma_2}$ of~$\g{g}$ are equal. To prove this, we consider separately the two cases where the Lie algebra~$\g{g}$ is absolutely simple and where it is not absolutely simple, since the Dynkin indices of subalgebras are defined differently in these two cases.
Note that the Killing form of a real form of a complex Lie algebra is given by restricting the Killing form of its complexification, see~\cite[\S2, Proposition~2]{Onisbook}. On the other hand, the Killing form of the realification~$\g{h}_\R$ of a complex Lie algebra~$\g{h}$ is given by
\[
\mathcal{B}_{\g{h}_\R}(X,Y) = 2 \Re (\mathcal{B}_{\g{h}}(X,Y)).
\]
First assume that $\g{g}$ is absolutely simple; then the Dynkin index of~$\g{g}_{\Sigma_j}$ is defined as the Dynkin index of the subalgebra~$(\g{g}_{\Sigma_j})_\C$ in~$\g{g}_\C$. In this case, the Lie algebras $\g{g}$, $\g{g}_{\Sigma_1}$, $\g{g}_{\Sigma_2}$ are real forms of the Lie algebras $\g{g}_\C$, $(\g{g}_{\Sigma_1})_\C$, $(\g{g}_{\Sigma_2})_\C$ and hence their Killing forms are given by restricting the Killing forms of their complexifications.
When $\g{g}_{\Sigma_1}$ and $\g{g}_{\Sigma_2}$ have simple complexifications, it is immediately clear that $\ind_{D}(\g{g}_{\Sigma_1})=\ind_{D}(\g{g}_{\Sigma_2})$ if and only if $c_1=c_2$.
In case the complexifications are not simple, we know by Lemma~\ref{lemma:eqDI} that they both are a direct sum of two isometric subalgebras of~$\g{g}_\C$ and it follows that $c_1=c_2$ if and only if their Dynkin indices agree.

Now assume that $\g{g}$ is not absolutely simple, i.e.\ it is the realification of a simple complex Lie algebra. In this case, the Dynkin index of~$\g{g}_{\Sigma_j}$ is defined as the Dynkin index of the subalgebra $\g{g}_{\Sigma_j} + i \g{g}_{\Sigma_j}$ of~$\g{g}$, viewed as a complex Lie algebra.
Now,  by hypothesis, the Lie algebras $\g{g}_{\Sigma_1}\simeq\g{g}_{\Sigma_2}$ is either a complex subalgebra or a real form. In the first case,
$\g{g}_{\Sigma_j}$ is a simple subalgebra of~$\g{g}$, viewed as a complex Lie algebra and and we conclude again that $c_1=c_2$ if and only the $\ind_D(\g{g}_{\Sigma_1}) = \ind_D(\g{g}_{\Sigma_2})$. In the second case, the result follows by Lemma~\ref{lemma:congrealforms}.
\end{proof}
\begin{definition}\label{def:Ltilde}
	Let $\g{g}$ be the Lie algebra of the isometry group of an irreducible symmetric space of non-compact type.  We define $\mathcal{L}(\g{g})$ as the set of isometry classes of maximal subalgebras of non-compact type of~$\g{g}$.
\end{definition}

Notice that this set is well defined since maximal subalgebras of non-compact type of $\g{g}$, when $\g{g}$ is complex, are either complex subalgebras or real forms of $\g{g}$ by Lemma~\ref{lemma:maximalcomplexsemisimple}.
We denote the isometry classes of subalgebras by the pairs $(\g{h},\ind_D(\g{h}))$, where the first entry denotes the isomorphism class of the subalgebra and the second entry is its Dynkin index.

\begin{corollary}
	\label{cor:correspondence}
	Let $\g{g}$ be the Lie algebra of the isometry group of an irreducible symmetric space of non-compact type~$M$.
	The set $\mathcal{L}(\g{g})$ is in one-to-one correspondence with the set of isometry classes of maximal semisimple totally geodesic submanifolds in~$M$.
\end{corollary}
\begin{proof}
	This follows from the Correspondence Theorem~\ref{th:correspondence} and Proposition~\ref{th:isom}.
\end{proof}

\section{Totally geodesic submanifolds in\\ exceptional symmetric spaces of type~III}
\label{sect:totgeodreal}

In this section, we will achieve the classification in the case where $\g{g}$ is the Lie algebra of the isometry group of an exceptional symmetric space of type III.
To the best of our knowledge, so far no complete classification of maximal subalgebras of non-compact type is available in the literature. There are, however, the articles of Komrakov~\cite{Komrakov} and de Graaf and Marrani~\cite{Degraaf}, where lists are given, which, when combined, provide all maximal subalgebras of non-compact type of a real semisimple exceptional Lie algebra~$\g{g}$. Indeed, on the one hand, the more recent paper of de Graaf and Marrani~\cite{Degraaf}, for absolutely simple real Lie algebras~$\g{g}$ of rank less or equal than~$8$, complete lists are given of maximal reductive subalgebras~$\g{h}$ whose complexification~$\g{h}_\C$ is a maximal reductive subalgebra of~$\g{g}_\C$.
On the other hand, in the article of Komrakov~\cite{Komrakov}, all maximal subalgebras~$\g{h}$ of non-compact simple real Lie algebras are given for the case when the complexification~$\g{h}_\C$ of~$\g{h}$ is not a maximal subalgebra of~$\g{g}_\C$.

	We consider ${\widetilde{\mathcal{L}}}(\g{g})$, the result of joining these two lists, deleting the compact ideals of each subalgebra in the union of these lists, and adding the Dynkin index of the resulting subalgebras, which we take from~\cite[Tables~16, 17, 18, 19, 20, 25]{Dynkin}.
For the absolutely simple exceptional real Lie algebras, the set~$\widetilde{\mathcal{L}}(\g{g})$ is given by Table~\ref{table:tildeL(g)}.

\begin{proposition}
	\label{prop:noncompactsubalgebras}
	Let $\g{g}$ be the Lie algebra of the isometry group of an irreducible exceptional symmetric space of non-compact type. Then, if $\g{g}$ is absolutely simple, $\mathcal{L}(\g{g})\subset \widetilde{\mathcal{L}}(\g{g})$.
\end{proposition}
\begin{proof}
	Let $\g{h}\subset\g{g}$ be a maximal subalgebra of non-compact type. Then, since $\g{h}$ is semisimple, $\g{h}$ is a  reductive algebraic subalgebra by Theorem~\ref{th:karpelevichalg} and Theorem~\ref{th:karpelevichmostow}. Hence, there exists a maximal reductive algebraic subalgebra $\widetilde{\g{h}}\subset\g{g}$ such that $\g{h}\subset \widetilde{\g{h}}$.
	Observe that $\widetilde{\g{h}}=\widetilde{\g{h}}_1 \oplus \widetilde{\g{h}}_2$, where $\widetilde{\g{h}}_1$ is a subalgebra of non-compact type and $\widetilde{\g{h}}_2$ is a  compact subalgebra.

Let us show that $\g{h}\subset \widetilde{\g{h}}_1$. Let $\widehat{\g{h}}$ be a simple ideal of $\g{h}$ and consider $f\colon \widetilde{\g{h}}\rightarrow\widetilde{\g{h}}_2$, the natural projection of $\widetilde{\g{h}}$ on $\widetilde{\g{h}}_2$. Then $f$ restricted to~$\widehat{\g{h}}$ is injective or zero. If it is injective, $\widetilde{\g{h}}_2$ has a subalgebra of non-compact type, which leads to a contradiction.

By definition, $(\widetilde{\g{h}})_\C\subset\g{g}_\C$ is a reductive algebraic subalgebra.
	On the one hand, if $(\widetilde{\g{h}})_\C\subset\g{g}_\C$ is not a maximal reductive subalgebra, in particular it is not maximal and it is given in~\cite[Theorem 1]{Komrakov}. On the other hand, if $(\widetilde{\g{h}})_\C\subset\g{g}_\C$ is a maximal reductive subalgebra, it is given in~\cite{Degraaf}. This shows ${\mathcal{L}}(\g{g})\subset{\widetilde{\mathcal{L}}}(\g{g})$.
\end{proof}

We prove a criterion which ensures the non-maximality of certain totally geodesic orbits.

\begin{lemma}\label{lemma:nonmaxtg}
Let $M = G/K$ be an irreducible symmetric space of non-compact type. Let $\g{h} = \g{k}_1 \oplus \g{g}_1$ be a semisimple subalgebra of~$\g{g}$, which is the direct sum of the ideals $\g{k}_1$ and $\g{g}_1$, where $\g{k}_1$ is compact and $\g{g}_1$ is of non-compact type.
Assume $\g{h}$ is canonically embedded into~$\g{g}$ with respect to the Cartan decomposition $\g{g}=\g{k}\oplus\g{p}$. Then the following holds:
\begin{enumerate}[i)]
  \item The linear subspace~$\ell := \g{h} \cap \g{p}$ of~$\g{p}$ is a Lie triple system and $L := \exp_o(\ell)$ is a totally geodesic orbit of the $H$-action on~$M$, where $\exp_o$ denotes the Riemannian exponential map of~$M$ at the point~$o$.
  \item The connected Lie subgroup~$K_1$ of~$K$ corresponding to~$\g{k}_1$ acts effectively on the normal space~$N_o(H \cdot o)$ to the $H$-orbit at~$o$ and trivially on~$\ell$.
  \item If $K_1$ contains a subgroup~$Q$ and there are vectors~$v,w \in N_o(H \cdot o)$ such that $K_1 \cdot v \neq v$, $Q \cdot v = v$, $Q \cdot w \neq w$, then $L$ is not a maximal totally geodesic submanifold of~$M$.
  \item If~$K_1$ is of rank greater than one or if the action of~$K_1$ on~$N_o(H \cdot o)$ is effectively an~$\mathsf{SO}_3$-representation, then $L$ is not a maximal totally geodesic submanifold of $M$.
\end{enumerate}
\end{lemma}

\begin{proof}
Since $\g{h}$ is canonically embedded into~$\g{g}$, we have that $\g{h} = (\g{h} \cap \g{k}) \oplus \ell$ is a Cartan decomposition of~$\g{h}$. This shows that $\ell$ is a Lie triple system in~$\g{p}$ and hence its exponential image is a totally geodesic submanifold of~$M$.  This proves \emph{i)}.

Observe that $K$ acts effectively on~$M$ since it is a subgroup of the isometry group of $M$. Every isometry~$f$ of~$M$ is uniquely defined by the image~$f(p)$ and the differential~$f_{*p}$ for one arbitrary point~$p \in M$. In particular, the action of an element in the isotropy group~$K=G_o$ on~$M$ is uniquely determined by the action of its differential on~$\g{p}=T_oM$. Therefore, if $k \in K$ acts trivially on~$\g{p}$, then $k = e$.
Since $\g{h}$ is canonically embedded, it is invariant under $\theta$. Thus the restriction of $\theta$ to~$\g{h}$ is a Lie algebra automorphism and it follows that $\theta$ leaves invariant $\g{k}_1$ and $\g{g}_1$, since $\g{k}_1$ is compact and $\g{g}_1$ is of non-compact type. Hence $\g{k}_1$ and $\g{g}_1$ are canonically embedded in~$\g{g}$ with respect to the Cartan decomposition $\g{g}=\g{k}\oplus\g{p}$ and it follows that
$\g{h}\cap\g{k}=\g{k}_1\oplus(\g{g}_1\cap\g{k})$ and $\g{h}\cap \g{p}=\g{g}_1\cap \g{p}=\ell$.
In particular, the subspaces $\g{k}_1$ and $\ell$ of $\g{g}$ commute. Hence, the connected subgroup $K_1$ acts trivially on~$\ell$ and effectively on $N_o(H\cdot o)$, since $K_1$ is a subgroup of the isometry group of $M$. This proves \emph{ii)}.

Assume there is a subgroup~$Q \subset K_1$ as described in part~\emph{iii)} of the assertion. Let $N$ be the connected component of the fixed point set of the $Q$-action on~$M$ containing $o$. Then $N$ is a totally geodesic submanifold of~$M$ containing~$o$ and $N=\exp_o(F)$, where $F$ is the set of fixed points under the action of $Q$ on~$\g{p}$. It follows that $F$ is a Lie triple system in~$\g{p}$, which contains~$\ell$, since $K_1$ acts trivially on~$\ell$ by~\emph{ii)}. Since there are vectors~$v,w \in N_o(H \cdot o)$ such that $K_1 \cdot v \neq v$, $Q \cdot v = v$, $Q \cdot w \neq w$, it follows that the dimension of~$F$ is strictly greater than the dimension of~$\ell$ and that~$F \neq \g{p}$. Hence $N$ is a proper totally geodesic submanifold containing~$L$ and $N \neq M$. This proves \textit{iii)}.

If the action of~$K_1$ on~$N_o(H \cdot o)$ is effectively an~$\mathsf{SO}_3$-representation then it is a direct sum of odd-dimensional irreducible real representations of~$K_1$ and a trivial module. Since non-trivial real representations of the abelian group~$\mathsf{SO}_2$ are even-dimensional, it follows that if we choose $Q = \mathsf{SO}_2$, then there are vectors~$v,w \in N_o(H \cdot o)$ as in item~\emph{iii)} of the statement.

Let us now assume that the rank of~$K_1$ is greater than one. First note that for the existence of a subgroup~$Q$ as described in~\emph{iii)}, it suffices that the action of~$K_1$ on~$\g{p}$ has a singular orbit which is not a fixed point of the $K_1$-action on~$\g{p}$. Indeed, let $v \in \g{p}$ be contained in such a singular orbit and let $Q$ be the isotropy subgroup~$(K_1)_v$. Then $K_1 \cdot v \neq v$ and $Q \cdot v = v$. If there is no vector $w \in N_o(H \cdot o)$ such that $Q \cdot w \neq w$, then all isotropy groups contain~$Q$, which implies that $Q$ is a principal isotropy group, a contradiction.

We may now assume that all singular orbits of the action of~$K_1$ on~$\g{p}$ are fixed points.
This implies that, as a $K_1$-module, $\g{p} = V_0 \oplus V_1$ is a direct sum of a trivial module~$V_0$ and a module~$V_1$ on which $K_1$~acts in such a fashion that each non-zero vector lies in a principal orbit. Then $K_1$ acts  on the unit sphere in~$V_1$ in such a way that all orbits are of the same dimension. In particular, the orbits comprise a regular homogeneous foliation of the unit sphere in~$V_1$. It follows from~\cite[Theorem~1.1]{Lu} that if this foliation is non-trivial (i.e.\ the action is not transitive on the unit sphere), the orbits are either one-dimensional or three-dimensional and that $K_1$, which acts effectively on~$\g{p}$, is isomorphic to~$\mathsf{U}_1$ or~$\mathsf{SU}_2$. Since we have excluded these cases by our hypothesis, we may from now on assume that the action is transitive on the unit sphere.

Hence $K_1$ is a connected Lie group acting effectively and transitively on a unit sphere. The corresponding presentations of spheres are, by~\cite{MontSam},
\begin{align*}
\mathrm{S}^n &= {\mathsf{SO}_{n+1}}/{\mathsf{SO}_n},\quad
\mathrm{S}^{2n+1} = {\mathsf{U}_{n+1}}/{\mathsf{U}_n} = {\mathsf{SU}_{n+1}}/{\mathsf{SU}_n},
\\
\mathrm{S}^{4n+3} &= {\mathsf{Sp}_{n+1}\mathsf{Sp}_1 }/{ \mathsf{Sp}_n\mathsf{Sp}_1} = {\mathsf{Sp}_{n+1}\mathsf{U}_1}/{\mathsf{Sp}_n\mathsf{U}_1} = {\mathsf{Sp}_{n+1}}/{\mathsf{Sp}_n},
\\
\mathrm{S}^{6} &= {\mathsf{G}_2}/{\mathsf{SU}_3},
\quad
\mathrm{S}^{7} = {\mathsf{Spin}_7}/{\mathsf{G}_2},
\quad
\mathrm{S}^{15} = {\mathsf{Spin}_9}/{\mathsf{Spin}_7}.
\end{align*}

By inspection of this list, we see that if the rank of~$K_1$ is two or greater, we may choose~$Q$ as the isotropy subgroup of any non-zero vector~$v \in \g{p}$; then $K_1 \cdot v \neq v$ and $Q \cdot v = v$ and $Q$ acts non-trivially on the orthogonal complement of~$v$ in~$\g{p}$, showing there is a vector~$w \in \g{p}$ such that $Q \cdot w \neq w$. We have completed the proof of~\emph{iv)}.\qedhere
\end{proof}

\begin{table}[p]\caption{$\widetilde{\mathcal{L}}(\g{g})$ for each exceptional absolutely simple real Lie algebra~$\g{g}$. We include the deleted compact ideals for the convenience of the reader.}
	\label{table:tildeL(g)}
	\begin{tabular}{cl}	
		\hline
		\multirow{1}{*}{$\g{g}^2_2$} & 	$(\g{sl}_3(\R),1),(\g{su}_{1,2},1),(\g{sl}_2(\R)\oplus\g{sl}_2(\R),(3,1)),(\g{sl}_2(\R),28)$
		\\ \hline
		\multirow{3}{*}{$\g{f}^4_{4}$} & 	$(\g{sl}_3(\R)\oplus\g{sl}_3(\R),(1,2)),(\g{su}_{1,2}\oplus\g{su}_{1,2},(1,2)),(\g{su}_{1,2}\oplus\cancel{\g{su}_3},2),$
		\\
		&$(\g{sl}_2(\R)\oplus\g{sp}_3(\R),(1,1)),(\g{sp}_{1,2}\oplus\cancel{\g{su}_2},1),(\g{so}_{4,5},1),(\g{sl}_{2}(\R),156),$
		\\& $(\g{sl}_2(\R)\oplus\g{g}_2^2,(8,1)),(\g{g}_2^2\oplus\cancel{\g{su}_2},1)$ \\
		\hline
		\multirow{1}{*}{$\g{f}^{-20}_{4}$} & 	$(\g{su}_{1,2} \oplus\cancel{\g{su}_3},2),(\g{sp}_{1,2}\oplus\cancel{\g{su}_2},1),(\g{so}_{1,8},1),(\g{sl}_2(\R)\oplus\cancel{\g{g}_2},8)$
		\\ \hline
		\multirow{3}{*}{$\g{e}^6_{6}$} & 	$(\g{so}_{5,5}\oplus\cancel{\R},1),(\g{sl}_{3}(\R)\oplus\g{sl}_{3}(\R)\oplus\g{sl}_{3}(\R),(1,1,1)),(\g{su}_{1,2}\oplus\g{sl}_3(\C),(1,1,1)),$
		\\
		&$(\g{sl}_2(\R)\oplus\g{sl}_6(\R),(1,1)),(\g{su}^*_{6}\oplus\cancel{\g{su}_2},1),(\g{sl}_{3}(\R)\oplus\g{g}^2_2,(2,1)),(\g{sp}_{2,2},1),$
		\\
		&$(\g{sp}_4(\R),1),(\g{f}_4^4,1)$
		\\ \hline
		\multirow{3}{*}{$\g{e}^2_{6}$} & $(\g{so}_{4,6}\oplus\cancel{\R},1),(\g{so}^*_{10}\oplus\cancel{\R},1),(\g{sl}_{3}(\C)\oplus\g{sl}_3(\R),(1,1,1)),(\g{su}_{1,2}\oplus\cancel{\g{su}_3}\oplus\cancel{\g{su}_3},1),$	
		\\
		&$(\g{su}_{1,2}\oplus\g{su}_{1,2}\oplus\g{su}_{1,2},(1,1,1)),(\g{sl}_2(\R)\oplus\g{su}_{3,3},(1,1)),(\g{su}_{2,4}\oplus\cancel{\g{su}_2},1),(\g{su}_{1,2},9),$\\
		&$(\g{sl}_3(\R),9),(\g{g}^2_2\oplus\cancel{\g{su}_3},1),(\g{g}_2^2,3),(\g{su}_{1,2}\oplus\g{g}_2^2,(2,1)),(\g{sp}_{1,3},1),(\g{sp}_4(\R),1),(\g{f}^4_4,1)$\\
		\hline
		\multirow{2}{*}{$\g{e}^{-14}_{6}$} &$(\g{so}^*_{10}\oplus\cancel{\R},1),(\g{so}_{2,8}\oplus\cancel{\R},1),(\g{su}_{1,2}\oplus\g{su}_{1,2}\oplus\cancel{\g{su}_3},(1,1)),(\g{sl}_2(\R)\oplus\g{su}_{1,5},(1,1)),$
		\\
		 &$(\g{su}_{2,4}\oplus\cancel{\g{su}_2},1),(\g{su}_{1,2}\oplus\cancel{\g{g}_2},2),(\g{sp}_{2,2},1),(\g{f}^{-20}_4,1)$\\ \hline
		\multirow{2}{*}{$\g{e}^{-26}_{6}$} &$(\g{so}_{1,9}\oplus\cancel{\R},1),(\g{sl}_{3}(\C)\oplus\cancel{\g{su}_3}, (1,1)),(\g{su}^*_{6}\oplus\cancel{\g{su}_2},1),(\g{sl}_3(\R)\oplus\cancel{\g{g}_2},2),$
		\\
		&$(\g{sp}_{1,3},1),(\g{f}^{-20}_4,1)$
		 \\ \hline
		\multirow{4}{*}{$\g{e}^{7}_{7}$} &$(\g{e}^{6}_6\oplus\cancel{\R},1),(\g{e}^{2}_6\oplus\cancel{\R},1),(\g{sl}_2(\R)\oplus\g{so}_{6,6},(1,1)),(\g{so}^*_{12}\oplus\cancel{\g{su}_2},1),(\g{sl}_8(\R),1),$\\
		&$(\g{su}_{4,4},1),(\g{su}^*_8,1),(\g{sl}_3(\R)\oplus\g{sl}_{6}(\R),(1,1)),(\g{su}_{1,2}\oplus\g{su}_{3,3},(1,1)),(\g{su}_{1,5}\oplus\cancel{\g{su}_3},1),$\\
		&$(\g{sl}_2(\R),231),(\g{sl}_2(\R),399),(\g{sl}_3(\R),21),(\g{sl}_2(\R)\oplus\g{sl}_{2}(\R),(15,24))$\\
		&$(\g{sl}_2(\R)\oplus\cancel{\g{g}_2},7),(\g{sl}_2(\R)\oplus\g{g}^2_2,(7,2)),(\g{sp}_3(\R)\oplus\g{g}^2_2,(1,1)),(\g{sl}_2(\R)\oplus\g{f}^{4}_4,(3,1))$
		\\ \hline
		\multirow{3}{*}{$\g{e}^{-25}_{7}$} &$(\g{e}^{-26}_6\oplus\cancel{\R},1),(\g{e}^{-14}_6\oplus\cancel{\R},1),(\g{sl}_2(\R)\oplus\g{so}_{2,10},(1,1)),(\g{so}^*_{12}\oplus\cancel{\g{su}_2},1),(\g{su}^*_8,1),$
		\\	
		& $(\g{su}_{2,6},1),(\g{su}_{3,3}\oplus\cancel{\g{su}_3},1),(\g{su}_{1,2}\oplus\g{su}_{1,5},(1,1)),(\g{sp}_3(\R)\oplus\cancel{\g{g}_2},1),(\g{sl}_2(\R)\oplus\cancel{\g{f}_4},3),$
		\\
		& $(\g{sl}_2(\R)\oplus\g{f}^{-20}_4,(3,1))$
		\\ \hline
		\multirow{4}{*}{$\g{e}^{-5}_{7}$} &$(\g{e}^{2}_6\oplus\cancel{\R},1),(\g{e}^{-14}_6\oplus\cancel{\R},1),(\g{sl}_2(\R)\oplus\g{so}^*_{12},(1,1)),(\g{so}_{4,8}\oplus\cancel{\g{su}_2},1),(\g{su}_{4,4},1),$	
		\\	
		& $(\g{su}_{2,6},1),(\g{sl}_3(\R)\oplus\g{su}^*_6,(1,1)),(\g{su}_{1,2}\oplus\g{su}_{2,4},(1,1)),(\g{su}_{2,4}\oplus\cancel{\g{su}_3},1),$\\
		&$(\g{su}_{1,2}\oplus\cancel{\g{su}_6},1),(\g{su}_{1,2},21),(\g{sl}_{2}(\R)\oplus\cancel{\g{su}_2},24),(\g{g}^2_2\oplus\cancel{\g{su}_2},2),(\g{g}^2_2\oplus\cancel{\g{sp}_3},1),$
		\\
		&
		$(\g{sp}_{1,2}\oplus\cancel{\g{g}_2},1),(\g{sp}_{1,2}\oplus\g{g}^2_2,(1,1)),(\g{f}^{4}_4\oplus\cancel{\g{su}_2},1),(\g{f}^{-20}_4\oplus\cancel{\g{su}_2},1)$
		\\ \hline
		\multirow{6}{*}{$\g{e}^{8}_{8}$} &$(\g{su}_{1,4}\oplus\cancel{\g{su}_5},1),(\g{su}_{1,4}\oplus\g{su}_{1,4},(1,1)),(\g{su}_{2,3}\oplus\g{su}_{2,3},(1,1)),(\g{sl}_3(\R)\oplus \g{e}^6_6,(1,1)),$\\	
		& $(\g{sl}_5(\R)\oplus\g{sl}_5(\R),(1,1)),(\g{so}_{8,8},1),(\g{so}^*_{16},1),(\g{su}_{1,8},1),(\g{su}_{4,5},1),(\g{sl}_9(\R),1),$\\
		&$(\g{e}^{-14}_6\oplus\cancel{\g{su}_3},1),(\g{su}_{1,2}\oplus\g{e}^2_6,(1,1)),(\g{e}^{-5}_7\oplus\cancel{\g{su}_2},1),(\g{sl}_2(\R)\oplus\g{e}^7_7,(1,1)),(\g{so}_{2,3},12),$\\
		& $(\g{so}_{1,4},12),(\g{su}_{1,2}\oplus\cancel{\g{su}_2},6),(\g{sl}_2(\R)\oplus\cancel{\g{su}_3},16), (\g{sl}_2(\R)\oplus\g{su}_{1,2},(16,6)),$\\
		& $(\g{sl}_2(\R)\oplus\g{sl}_3(\R),(16,6)),(\g{f}^{-20}_4\oplus\cancel{\g{g}_2},1),(\g{f}^4_4\oplus\g{g}^2_2,(1,1)),$
		\\
		&
		$(\g{sl}_2(\R),520),(\g{sl}_2(\R),760),(\g{sl}_2(\R),1240),(\g{g}_2(\C)\oplus\g{sl}_2(\R),(1,1,8))$
		\\ \hline
		\multirow{4}{*}{$\g{e}^{-24}_{8}$} &$(\g{su}_{1,4}\oplus\g{su}_{2,3},(1,1)),(\g{su}_{2,3}\oplus\cancel{\g{su}_5},1),(\g{so}_{4,12},1),(\g{so}^*_{16},1),(\g{su}_{3,6},1),(\g{su}_{2,7},1),$\\	
		& $(\g{e}^2_6\oplus\cancel{\g{su}_3},1),(\g{su}_{1,2}\oplus\g{e}^{-14}_6,(1,1)),(\g{su}_{1,2}\oplus\cancel{\g{e}_6},1),(\g{sl}_3(\R)\oplus\g{e}^{-26}_6,(1,1)),$\\
		&$(\g{e}^{-5}_7\oplus\cancel{\g{su}_2},1),(\g{sl}_2(\R)\oplus\g{e}^{-25}_7,(1,1)),(\g{sl}_3(\R)\oplus\cancel{\g{su}_2},6),(\g{f}^4_4\oplus\cancel{\g{g}_2},1),$
		\\
		& $(\g{g}_2^2\oplus\cancel{\g{f}_4},1),(\g{f}^{-20}_4\oplus\g{g}^2_2,(1,1)),(\g{g}_2(\C)\oplus\cancel{\g{su}_2},(1,1))$
		\\ \hline
	\end{tabular}
\end{table}
Observe that while many subalgebras of $\widetilde{\mathcal{L}}(\g{g})$ are obtained by removing a compact ideal $\g{k}_1$ of rank greater than one, hence satisfying the hypotheses of Lemma~\ref{lemma:nonmaxtg}~\textit{iv)}, there are still some subalgebras of $\widetilde{\mathcal{L}}(\g{g})$ which cannot be directly treated using this lemma. In order to deal with these cases, we prove the following results.
\begin{lemma}
	\label{lemma:exceptions}
	Let $\g{g}$ be an exceptional absolutely simple non-compact real Lie algebra.
We consider the following maximal subalgebras $\g{h}$ of real Lie algebras~$\g{g}$ which are the non-compact parts of some subalgebras of $\g{g}$ with a rank one simple compact factor appearing in~$\mathcal{\widetilde{L}}(\g{g})$:
	\begin{align*}
	\g{g}_2^2 \subset \g{f}_4^4; \quad
	\g{sl}_2(\R),\;  \g{g}_2^2,\; \g{f}_4^4,\; \g{f}_4^{-20} \subset \g{e}_7^{-5};\quad
	\g{g}_2(\C),\;  \g{sl}_3(\R) \subset \g{e}_8^{-24};  \quad
	 \g{su}_{1,2} \subset \g{e}_8^8.
	\end{align*}
	Let $G$ be the simply connected Lie group with Lie algebra~$\g{g}$ and let $K$ be a  maximal compact subgroup of~$G$. Assume the Lie algebra~$\g{h}$ is canonically embedded into~$\g{g}$ with respect to the Cartan decomposition~$\g{g} = \g{k} \oplus \g{p}$. Let $H$ be the connected closed subgroup of~$G$ with Lie algebra~$\g{h}$. Then a totally geodesic orbit of the $H$-action on the symmetric space~$M = G/K$ is not maximal.
\end{lemma}

\begin{proof}
	The complexifications of all the subalgebras~$\g{h}$ in the statement of the lemma can be found in~\cite[Table~39, p.~233]{Dynkin}, where we can also read off the Dynkin indices (given as superscripts in the table) of (the complexifications of) their compact ideals (which are all isomorphic to~$\g{su}_2$).
	The so-called characteristic representations of the $\g{su}_2$-summands, i.e.\ the representations $\ad_{\g{g}} \ominus \ad_{\g{su}_2}$ are given in~\cite[Table~21, p.~186-187]{Dynkin}. It follows from this table that they all are direct sums of odd-dimensional irreducible $\g{su}_2$-representations and a trivial module. Hence, on the level of Lie groups, the corresponding representations are effectively~$\mathsf{SO}_3$-representations.
	Thus the subgroup~$K_1$ as described in Lemma~\ref{lemma:nonmaxtg} effectively acts as~$\mathsf{SO}_3$ on~$\g{p}$ and it follows from Lemma~\ref{lemma:nonmaxtg} \textit{iv)} that a totally geodesic orbit of the $H$-action on~$M$ is not maximal.
\end{proof}

\begin{lemma}
	\label{lemma:nonsemi}
	Consider the following subalgebras $\g{h}$ of real Lie algebras~$\g{g}$ which are the semisimple parts of some subalgebras of $\g{g}$ with one dimensional center appearing in~$\mathcal{\widetilde{L}}(\g{g})$:
	\begin{align*}
	\g{so}_{5,5}  &\subset \g{e}_6^6, \qquad\; \g{so}_{9,1} \subset \g{e}_6^{-26},
	\\
	\g{e}^6_6 &\subset \g{e}^7_7, \qquad\; \g{e}^{-26}_6 \subset \g{e}_7^{-25}.\;
	\end{align*}
	Let $G$ be the simply connected Lie group with Lie algebra~$\g{g}$ and let $K$ be a maximal compact subgroup of~$G$. Assume the Lie algebra~$\g{h}$ is canonically embedded into~$\g{g}$ with respect to the Cartan decomposition~$\g{g} = \g{k} \oplus \g{p}$. Let $H$ be the connected closed subgroup of~$G$ with Lie algebra~$\g{h}$. Then a totally geodesic orbit of the $H$-action on the symmetric space~$M = G/K$ is not maximal.
\end{lemma}
\begin{proof}
First note that $\g{h}\oplus\g{s}$, where $\g{s}$ is an abelian 1-dimensional subalgebra of~$\g{g}$, are reductive algebraic subalgebras. This is because their complexifications are maximal subalgebras of maximal rank of $\g{g}_\C$. Hence, they are algebraic by Remark~\ref{rem:alg}. By observing that the root spaces of a simple complex Lie algebra are equal to the complexification of the root spaces of some compact real form, and by using the Borel-de~Siebenthal theorem, see~\cite{boreldesiebenthal}, we deduce that they are regular subalgebras of $\g{g}_\C$. However, regular subalgebras are clearly canonically embedded with respect to some Cartan decomposition of $\g{g}_\C$ and then these are reductive algebraic by Theorem~\ref{th:karpelevichmostow}.

Now, notice that $\g{s}$ is either contained in~$\g{p}$ or in~$\g{k}$. Indeed, $\theta$, the Cartan involution of $\g{g}$, restricted to~$\g{h}\oplus\g{s}$ is a Lie algebra automorphism, hence it maps $\g{s}$ onto~$\g{s}$, since $\g{s}$ is the center of $\g{h}\oplus\g{s}$. However, a one dimensional subspace  invariant under $\theta$ must be either contained in~$\g{p}$ or be contained in~$\g{k}$. Let us prove that $\g{s}$ is contained in~$\g{p}$ in all four cases. We will argue by contradiction, so we assume that $\g{s}\subset\g{k}$.

Let $\g{g}=\g{e}^6_6$, then $\g{k}\simeq \g{sp}_4$. If $\g{h}\simeq \g{so}_{5,5}$, then $\g{h}\cap\g{k}\simeq \g{sp}_2\oplus\g{sp}_2$. However, $\g{sp}_4$ does not contain a subalgebra isomorphic to $\g{sp}_2\oplus\g{sp}_2\oplus\R$, contradicting the assumption $\g{s}\subset\g{k}$.

Let $\g{g}=\g{e}^{-26}_6$, then $\g{k}\simeq \g{f}_4$. If $\g{h}\simeq\g{so}_{9,1}$, then $\g{h}\cap\g{k}\simeq \g{so}_9$. However, $\g{f}_4$ does not contain a subalgebra isomorphic to $\g{so}_9\oplus\R$, which contradicts the assumption $\g{s}\subset\g{k}$.

Let $\g{g}=\g{e}^7_7$, then $\g{k}\simeq \g{su}_8$.  If $\g{h}\simeq \g{e}^6_6$, then $\g{h}\cap \g{k}\simeq \g{sp}_4$. However, $\g{su}_8$ does not contain a subalgebra isomorphic to $\g{sp}_4\oplus\R$, again contradicting $\g{s}\subset\g{k}$.

Finally, let $\g{g}=\g{e}_7^{-25}$, then $\g{k}\simeq \g{e}_6\oplus\R$. Now, if $\g{h}\simeq \g{e}^{-26}_6$, then $\g{h}\cap\g{k}\simeq \g{f}_4$. In addition to that, observe that $\g{f}_4$ is maximal in~$\g{e}_6$, therefore we have that $\g{s}$ is equal to the abelian factor of~$\g{k}$. Moreover, the abelian ideal $\g{s}$ of $\g{h}\oplus\g{s}$ acts trivially on $\g{h}\cap\g{p}$. However, $\g{s}$ corresponds to the abelian factor of the isotropy of a Hermitian symmetric space, therefore it cannot act trivially on any non-trivial subspace of $\g{p}$.

In all these cases, we obtain a contradiction with our assumption $\g{s}\subset\g{k}$. Thus, $\g{s}\subset\g{p}$. Therefore, by Proposition~\ref{prop:congruency}, we have that every totally geodesic orbit induced by $\g{h}$ is properly contained in a  totally geodesic orbit induced by $\g{h}\oplus\g{s}$.
\end{proof}

\begin{theorem}
	\label{th:tablesreal}
	Let $M=G/K$  be a symmetric space of non-compact type where $G$ is an exceptional Lie group whose Lie algebra is absolutely simple. Let $\Sigma\subset M$ be a maximal totally geodesic submanifold. Then $\Sigma$ is isometric to one of the spaces listed in Tables~\ref{table:g2}, \ref{table:f4}, \ref{table:e6}, \ref{table:e7}, \ref{table:e8}. Conversely, every space~$\Sigma$ listed in these tables can be isometrically embedded as a maximal totally geodesic submanifold of~$M$.
\end{theorem}
\begin{proof}
Let $\Sigma\subset M$ be a maximal totally geodesic submanifold. If $\Sigma$ is a non-semisimple totally geodesic submanifold, it is one of the examples found by Berndt and Olmos~\cite{BO1}, which we include in our tables.

Now assume that $\Sigma$ is semisimple. Then, by Corollary~\ref{cor:correspondence}, there is a corresponding Lie subalgebra $\g{g}_{\Sigma}$ given by ${\mathcal{L}}(\g{g})$ and, by Proposition~\ref{prop:noncompactsubalgebras}, $\g{g}_{\Sigma}$ is isometric to some subalgebra given by $\widetilde{\mathcal{L}}(\g{g})$.

We have to decide which elements in~$\widetilde{\mathcal{L}}(\g{g})$ actually give rise to maximal totally geodesic submanifolds. Let us define
\[\widehat{\mathcal{L}}(\g{g}):=\{ (\g{h},\ind_D(\g{h}))\in\widetilde{\mathcal{L}}(\g{g}): \text{$\g{h}_\C$ is  a maximal reductive subalgebra of $\g{g}_\C$}\}.\]
By Lemma~\ref{lemma:complexification} and the definition of  $\widetilde{\mathcal{L}}(\g{g})$  the subalgebras given by $\widehat{\mathcal{L}}(\g{g})$ are exactly those of non-compact type in~\cite{Degraaf} and~\cite{Komrakov}, i.e.\ not containing compact ideals. Note that since subalgebras in~$\widehat{\mathcal{L}}(\g{g})$ are semisimple, it follows that they are maximal semisimple subalgebras by Lemma~\ref{lemma:complexification} \textit{ii)}.
By Corollary~\ref{cor:correspondence}, every subalgebra given by $\widehat{\mathcal{L}}(\g{g})$ induces a maximal semisimple totally geodesic submanifold of~$M$.
Let us assume that $(\g{g}_{\Sigma},\ind_D(\g{g}_{\Sigma}))\in\widehat{\mathcal{L}}(\g{g})$ does not induce a maximal totally geodesic submanifold. Then $\g{g}_{\Sigma}$ must be contained in a reductive non-semisimple subalgebra. However, $\g{g}_{\Sigma}$ is a maximal reductive subalgebra by Lemma~\ref{lemma:complexification} \textit{i)}, contradicting our assumption.
We have now shown that every subalgebra in~$\widehat{\mathcal{L}}(\g{g})$ induces a maximal totally geodesic submanifold and we include those submanifolds in our tables.

As a consequence, it now remains to deal with the subalgebras in the complement \[\mathcal{C}(\g{g}):=\widetilde{\mathcal{L}}(\g{g})\setminus\widehat{\mathcal{L}}(\g{g}).\] We consider each case separately. Note that every subalgebra in $\mathcal{C}(\g{g})$ is obtained from some larger subalgebra of~$\g{g}$ which contains a compact ideal, see Table~\ref{table:tildeL(g)}.

Let $\g{g}=\g{g}^2_2$ and $M=\mathsf{G}_2^2/\mathsf{SO}_4$. In this case we are done since $\mathcal{C}(\g{g})=\emptyset$.

Let $\g{g}=\g{f}^4_4$ and $M=\mathsf{F}^4_4/\mathsf{Sp}_3\mathsf{Sp}_1$.
In this case $\mathcal{C}(\g{g})=\{(\g{g}^2_2,1),(\g{su}_{1,2},2),(\g{sp}_{1,2},1)\}$. Moreover, observe that every subalgebra isomorphic to $\g{sp}_{1,2}$ is a maximal non-compact type subalgebra of  $\g{f}^4_4$. Indeed, if $\g{h}\simeq \g{sp}_{1,2}$, then $\g{h}_{\C}\simeq \g{sp}_3(\C)$. We can deduce by~\cite[Table 25, p.~199]{Dynkin} that $\g{sl}_2(\C)\oplus\g{sp}_3(\C)$ is the only proper semisimple subalgebra of $\g{f}_4(\C)$ in which $\g{h}_{\C}$ is properly contained. However, according to Table~\ref{table:tildeL(g)}, the only real forms in~$\g{g}$ for the embedding $\g{sl}_2(\C)\oplus \g{sp}_3(\C)\subset \g{f}_4(\C)$ are $\g{sl}_2(\R)\oplus \g{sp}_{3}(\R)$ and $\g{su}_2\oplus \g{sp}_{1,2}$. Thus, $\g{sp}_{1,2}$ is a maximal non-compact type subalgebra of~$\g{f}^4_4$. Then, the corresponding totally geodesic orbit is a maximal semisimple totally geodesic submanifold, which is maximal, since $M$ does not have non-semisimple maximal totally geodesic submanifolds by~\cite{BO1}.
Furthermore, by~Table~\ref{table:tildeL(g)}, a subalgebra isometric to $(\g{su}_{1,2},2)$  is contained in $\g{su}_3\oplus\g{su}_{1,2}$ or $\g{su}_{1,2}\oplus\g{su}_{1,2}$. Then in the former case, by Lemma~\ref{lemma:nonmaxtg} \emph{iv}), the corresponding totally geodesic orbit is not maximal, since $\rank(\g{su}_3)=2$ and in the latter case is obviously not maximal. The  totally geodesic orbit corresponding to $(\g{g}^2_2,1)$ is not maximal by Lemma~\ref{lemma:exceptions}.

Let $\g{g}=\g{f}^{-20}_4$ and $M=\mathsf{F}^{-20}_4/\mathsf{Spin}_9$. In this case  $\mathcal{C}(\g{g})=\{(\g{sp}_{1,2},1),(\g{su}_{1,2},2),(\g{sl}_2(\R),8)  \}$. Consider a subalgebra isometric to $(\g{sp}_{1,2},1)$.  Now, since its complexification cannot be embedded as a subalgebra of the complexification of any other subalgebra in~$\widetilde{\mathcal{L}}(\g{g})$, we have that $(\g{sp}_{1,2},1)\in\mathcal{L}(\g{g})$. Furthermore, by~\cite{BO1}, there are no non-semisimple maximal totally geodesic submanifolds in~$M$, so $(\g{sp}_{1,2},1)$ induces a maximal totally geodesic submanifold. On the other hand,  by Table~\ref{table:tildeL(g)}, every subalgebra isometric to $(\g{su}_{1,2},2)$ or $(\g{sl}_2(\R),8)$ is such that it can be embedded in the following subalgebras of $\g{g}$:
\[ \g{su}_{1,2}\subset \g{su}_3\oplus\g{su}_{1,2}, \quad \g{sl}_2(\R)\subset\g{sl}_2(\R)\oplus\g{g}_2.  \]
Thus, by Lemma~\ref{lemma:nonmaxtg} \emph{iv}), the corresponding totally geodesic orbits to these subalgebras are not maximal totally geodesic submanifolds.
	
Let $\g{g}=\g{e}^6_6$ and $M=\mathsf{E}^6_6/\mathsf{Sp}_4$. In this case $\mathcal{C}(\g{g}):=\{(\g{so}_{5,5},1),(\g{su}^*_{6},1)\}$.	Every subalgebra isometric to $(\g{su}^*_6,1)$ is a  maximal non-compact type subalgebra of $\g{g}$ since its complexification is not contained in the complexification of any other subalgebra of $\widetilde{\mathcal{L}}(\g{g})$ except for $\g{sl}_6(\R)$, which clearly does not contain~$\g{su}^*_6$. Thus, $(\g{su}^*_6,1)$  induces a maximal totally geodesic submanifold since the corresponding totally geodesic orbit is not contained in a non-semisimple totally geodesic submanifold in the list given by~\cite{BO1}. By Lemma~\ref{lemma:nonsemi}, the totally geodesic orbits corresponding to $\g{so}_{5,5}$ are not maximal totally geodesic submanifolds in~$M$.

Let $\g{g}=\g{e}_6^2$ and $M=\mathsf{E}^2_6/\mathsf{SU}_6\mathsf{Sp}_1$. In this case
 \[\mathcal{C}(\g{g}):=\{(\g{so}_{4,6},1),(\g{so}^*_{10},1),(\g{su}_{1,2},1),(\g{su}_{2,4},1),(\g{g}^2_2,1)\}.\]
By Table~\ref{table:tildeL(g)}, we have the following embeddings into subalgebras of $\g{g}$:
\[  \g{su}_{1,2}\subset \g{su}_{3}\oplus \g{su}_{3}\oplus\g{su}_{1,2}, \quad \g{g}^2_2\subset\g{su}_3\oplus\g{g}^2_2.   \]
Thus, by Lemma~\ref{lemma:nonmaxtg} \emph{iv}), we have that $(\g{su}_{1,2},1),(\g{g}^2_2,1)$ do not induce maximal totally geodesic submanifolds.
Moreover, the complexification of~$\g{su}_{2,4}$ cannot be contained in the complexification of any other subalgebra in~$\widetilde{\mathcal{L}}(\g{g})$ except for $(\g{sl}_2(\R)\oplus\g{su}_{3,3},(1,1))$, which clearly does not contain~$\g{su}_{2,4}$. Therefore, it induces a maximal totally geodesic submanifold since there are no non-semisimple maximal totally geodesic submanifolds in~$M$ according to ~\cite{BO1}. It follows from Table~\ref{table:tildeL(g)} that $(\g{so}_{4,6},1)$ and $(\g{so}^*_{10},1)$ can be properly contained in just one proper reductive subalgebra of~$\g{g}$. Namely,
\[\g{so}_{4,6}\subset \g{so}_{4,6}\oplus\R\qquad \g{so}^*_{10}\subset \g{so}^*_{10}\oplus\R.   \]
Thus, both subalgebras induce maximal totally geodesic submanifolds since there are no non-semisimple maximal totally geodesic submanifolds in~$M$.

Let $\g{g}=\g{e}^{-14}_6$ and $M=\mathsf{E}^{-14}_6/\mathsf{Spin}_{10}\mathsf{U}_1$. In this case \[\mathcal{C}(\g{g})=\{ (\g{so}^*_{10},1), (\g{so}_{2,8},1),(\g{su}_{1,2}\oplus\g{su}_{1,2},(1,1)),(\g{su}_{2,4},1),(\g{su}_{1,2},2)  \}.\]
Moreover, $(\g{su}_{2,4},1)$ is such that its complexification cannot be embedded as a subalgebra of the complexification of any other subalgebra in~$\widetilde{\mathcal{L}}(\g{g})$. Thus, $(\g{su}_{2,4},1)$ induces a maximal totally geodesic submanifold, since there are no non-semisimple maximal totally geodesic submanifolds in~$M$ according to~\cite{BO1}. Furthermore, by Lemma~\ref{lemma:nonmaxtg} \emph{iv}), we have that $(\g{su}_{1,2}\oplus\g{su}_{1,2},(1,1))$ and $(\g{su}_{1,2},2)$ do not induce maximal totally geodesic submanifolds since by Table~\ref{table:tildeL(g)}, we have the following inclusions into subalgebras of $\g{g}$:
\[\g{su}_{1,2}\oplus\g{su}_{1,2}\subset \g{su}_3\oplus\g{su}_{1,2}\oplus\g{su}_{1,2} \qquad \g{su}_{1,2}\subset\g{su}_{1,2}\oplus \g{g}_2.  \]
 Furthermore, by  Table~\ref{table:tildeL(g)}, $(\g{so}_{2,8},1)$ and $(\g{so}^*_{10},1)$ can be properly contained in just one proper reductive subalgebra of~$\g{g}$. Namely,
\[\g{so}_{2,8}\subset \g{so}_{2,8}\oplus\R\qquad \g{so}^*_{10}\subset \g{so}^*_{10}\oplus\R.   \]
However, according to~\cite{BO1}, there are no non-semisimple maximal totally geodesic submanifolds in~$M$. Thus, both subalgebras induce maximal totally geodesic submanifolds.

Let $\g{g}=\g{e}_6^{-26}$ and $M=\mathsf{E}^{-26}_6/\mathsf{F}_4$. In this case
\[ \mathcal{C}(\g{g})=\{(\g{so}_{1,9},1),(\g{sl}_3(\C),(1,1)),(\g{su}^*_6,1),(\g{sl}_3(\R),2)   \}. \]
Note that $(\g{su}^*_6,1)$ is such that its complexification cannot be embedded as a subalgebra of the complexification of any other subalgebra in~$\widetilde{\mathcal{L}}(\g{g})$. Thus, $(\g{su}^*_6,1)$ induces a maximal totally geodesic submanifold, since there is just one  non-semisimple maximal totally geodesic submanifold in the list given by~\cite{BO1}, which is $\R\times\mathsf{SO}^0_{1,9}/\mathsf{SO}_9$, and it clearly does not contain a totally geodesic $\mathsf{SU}^*_6/\mathsf{Sp}_3$. Furthermore, by Lemma~\ref{lemma:nonmaxtg} \emph{iv)}, we have that $(\g{sl}_3(\C),(1,1))$ and $(\g{sl}_3(\R),2)$ do not induce maximal totally geodesic submanifolds, since by Table~\ref{table:tildeL(g)}, we have the following inclusions into subalgebras of $\g{g}$:
\[ \g{sl}_3(\C)\subset\g{sl}_3(\C)\oplus\g{su}_3, \quad \g{sl}_3(\R)\subset \g{sl}_3(\R)\oplus\g{g}_2.  \]
Now by Lemma~\ref{lemma:nonsemi}, $(\g{so}_{1,9},1)$ does not induce maximal totally geodesic submanifolds.

Let $\g{g}=\g{e}^7_7$ and $M=\mathsf{E}^7_7/\mathsf{SU}_8$. In this case \[\mathcal{C}(\g{g})=\{(\g{e}^6_6,1),(\g{e}^2_6,1),(\g{so}^*_{12},1),(\g{su}_{1,5},1),(\g{sl}_2(\R),7)\}.\]
By Table~\ref{table:tildeL(g)}, we have the following inclusions into subalgebras of $\g{g}$:
\[\g{su}_{1,5}\subset \g{su}_3\oplus\g{su}_{1,5},\qquad \g{sl}_2(\R)\subset \g{sl}_2(\R)\oplus\g{g}_2.\]
Thus, by Lemma~\ref{lemma:nonmaxtg} \emph{iv)}, we have that $(\g{su}_{1,5},1)$ and $(\g{sl}_2(\R),7)$ do not induce maximal totally geodesic submanifolds. Furthermore, $(\g{e}^2_6,1)$ is such that its complexification is not contained in the complexification of any other subalgebra in~$\widetilde{\mathcal{L}}(\g{g})$ except for $\g{e}^{6}_6$, which clearly does not contain~$\g{e}^2_6$. Thus, $(\g{e}^2_6,1)$ gives a maximal totally geodesic submanifold, since by~\cite{BO1}, the only non-semisimple maximal totally geodesic submanifold is $\R\times \mathsf{E}^6_6/\mathsf{Sp}_4$, which cannot contain a totally geodesic $\mathsf{E}^2_6/\mathsf{SU}_6\mathsf{Sp}_1$. In addition to that, $(\g{so}^*_{12},1)$ is such that its complexification is not contained in the complexification of any  other subalgebra in~$\widetilde{\mathcal{L}}(\g{g})$ except for $\g{sl}_2(\R)\oplus\g{so}_{6,6}$, which clearly does not contain~$\g{so}^*_{12}$. Consequently, $(\g{so}^*_{12},1)$ induces a maximal totally geodesic submanifold in~$M$, since its corresponding totally geodesic submanifold cannot be totally geodesically embedded in $\R\times \mathsf{E}^6_6/\mathsf{Sp}_4$.   Finally, by Lemma~\ref{lemma:nonsemi}, we have that $\g{e}^6_6$ does not induce a maximal totally geodesic submanifold.
	
Let $\g{g}=\g{e}^{-25}_7$ and $M=\mathsf{E}^{-25}_7/\mathsf{E}_6\mathsf{U}_1$. In this case, we have that
\[\mathcal{C}(\g{g}):=\{(\g{e}^{-26}_6,1),(\g{e}^{-14}_6,1),(\g{so}^*_{12},1),(\g{su}_{3,3},1),(\g{sp}_3(\R),1),(\g{sl}_2(\R),3)  \}.  \]
Notice that $\g{e}^{-14}_6$ is such that its complexification is not contained in the complexification of any other subalgebra in~$\widetilde{\mathcal{L}}(\g{g})$ but $\g{e}^{-26}_6$ and clearly it is not contained in this one. Moreover, by~\cite{BO1}, the only non-semisimple maximal totally geodesic submanifold in~$M$ is $\R\times\mathsf{E}^{-26}_6/\mathsf{F}_{4}$, which cannot contain a totally geodesic $\mathsf{E}^{-14}_6/\mathsf{Spin}_{10}\mathsf{U}_1$.  This implies that $(\g{e}^{-14}_6,1)$ induces a maximal totally geodesic submanifold of~$M$. Furthermore, by Lemma~\ref{lemma:nonsemi}, $(\g{e}^{-26}_6,1)$ cannot induce a maximal totally geodesic submanifold of~$M$. In addition to that, $(\g{so}^*_{12},1)$ is such that its complexification is not contained in the complexification of any other subalgebra in~$\widetilde{\mathcal{L}}(\g{g})$ except for $\g{sl}_2(\R)\oplus\g{so}_{2,10}$, which clearly does not contain~$(\g{so}^*_{12},1)$. Thus, $(\g{so}^*_{12},1)$ induces a maximal totally geodesic submanifold, since it is not totally geodesic embedded in $\R\times\mathsf{E}^{-26}_6/\mathsf{F}_{4}$. Now, we have by Lemma~\ref{lemma:nonmaxtg} \emph{iv)} that $(\g{su}_{3,3},1)$, $(\g{sp}_3(\R),1)$ and $(\g{sl}_2(\R),3)$ cannot give totally geodesic submanifolds since by Table~\ref{table:tildeL(g)} they are contained in the following subalgebras of~$\g{g}$:
\[\g{sl}_2(\R)\subset\g{sl}_2(\R)\oplus\g{f}_4,\quad \g{su}_{3,3}\subset\g{su}_3\oplus\g{su}_{3,3},\quad \g{sp}_3(\R)\subset\g{sp}_3(\R)\oplus\g{g}_2.   \]

Let $\g{g}=\g{e}^{-5}_7$ and $M=\mathsf{E}^{-5}_7/\mathsf{SO}_{12}\mathsf{Sp}_1$. In this case we have
\begin{align*}
\mathcal{C}(\g{g})=\{&(\g{e}^2_6,1),(\g{e}^{-14}_6,1),(\g{so}_{4,8},1),(\g{su}_{2,4},1),(\g{su}_{1,2},1),(\g{sl}_2(\R),24),(\g{g}^2_2,2),(\g{g}^2_2,1),(\g{sp}_{1,2},1),\\
&(\g{f}^4_4,1),(\g{f}^{-20}_4,1) \}.
\end{align*}
Note that $(\g{e}^2_6,1)$ and $(\g{e}^{-14}_6,1)$ are such that their complexifications are not contained in the complexification of any other subalgebra in~$\widetilde{\mathcal{L}}(\g{g})$ except for each other. However, $\g{e}^2_6$ cannot be contained in~$\g{e}^{-14}_6$ and viceversa. By~\cite{BO1}, there are no non-semisimple maximal totally geodesic submanifolds in~$M$. Then  $(\g{e}^2_6,1)$ and $(\g{e}^{-14}_6,1)$ induce maximal totally geodesic submanifolds in~$M$. In addition to that, $(\g{so}_{4,8},1)$ is such that its complexification is not contained in the complexification of any other subalgebra in~$\widetilde{\mathcal{L}}(\g{g})$ except for $\g{sl}_2(\R)\oplus\g{so}^*_{12}$, which clearly cannot contain~$\g{so}_{4,8}$. Thus, $(\g{so}_{4,8},1)$ induces a maximal totally geodesic submanifold in~$M$. Furthermore, by Table~\ref{table:tildeL(g)}, we have the following inclusions into subalgebras of $\g{g}$:
\[ \g{su}_{2,4}\subset\g{su}_{3}\oplus\g{su}_{2,4}, \quad\g{su}_{1,2}\subset\g{su}_{1,2}\oplus\g{su}_{6},\quad \g{g}^2_2\subset \g{sp}_3\oplus\g{g}^2_2,\quad \g{sp}_{1,2}\subset\g{sp}_{1,2}\oplus\g{g}_2.\]
Therefore, by Lemma~\ref{lemma:nonmaxtg} \emph{iv}), we have that  $(\g{su}_{2,4},1)$, $(\g{su}_{1,2},1)$, $(\g{g}^2_2,1)$ and $(\g{sp}_{1,2},1)$ do not induce maximal totally geodesic submanifolds. Now by Lemma~\ref{lemma:exceptions}, we have that $(\g{sl}_2(\R),24)$, $ (\g{g}^2_2,2)$, $(\g{f}^4_4,1)$ and $(\g{f}^{-20}_4,1)$ do not induce maximal totally geodesic submanifolds.

Let $\g{g}=\g{e}^8_8$ and $M=\mathsf{E}^8_8/\mathsf{SO}_{16}$. In this case \[\mathcal{C}(\g{g})=\{(\g{su}_{1,4},1), (\g{e}^{-14}_6,1),(\g{e}^{-5}_7,1),(\g{su}_{1,2},6),(\g{f}^{-20}_4,1),(\g{sl}_2(\R),16)   \}.\]
Notice that $\g{e}^{-5}_7$ cannot be contained in any other subalgebra in~$\widetilde{\mathcal{L}}(\g{g})$. However, by~\cite{BO1}, there is no non-semisimple maximal totally geodesic submanifold in~$M$, which implies that $(\g{e}^{-5}_7,1)$ induces a maximal totally geodesic submanifold. Now, by Lemma~\ref{lemma:exceptions}, we have that $(\g{su}_{1,2},6)$ does not induce a maximal totally geodesic submanifold in~$M$. Furthermore, by Table~\ref{table:tildeL(g)}, we have the following inclusions into subalgebras of $\g{g}$:
\[ \g{su}_{1,4}\subset\g{su}_{1,4}\oplus\g{su}_5,\quad \g{e}^{-14}_6\subset \g{e}^{-14}_6\oplus\g{su}_3,\quad \g{f}^{-20}_4\subset\g{f}^{-20}_4\oplus \g{g}_2,\quad \g{sl}_2(\R)\subset\g{sl}_2(\R)\oplus \g{su}_3.   \]
Thus, by Lemma~\ref{lemma:nonmaxtg} \emph{iv)}, we have that $(\g{su}_{1,4},1)$, $(\g{e}^{-14}_6,1)$, $(\g{f}^{-20}_4,1)$ and $(\g{sl}_2(\R),16)$ do not induce maximal totally geodesic submanifolds in~$M$.

Let $\g{g}=\g{e}^{-24}_8$ and $M=\mathsf{E}^{-24}_8/\mathsf{E}_7\mathsf{Sp}_1$. In this case
\[\mathcal{C}(\g{g})=\{(\g{su}_{2,3},1),(\g{e}^2_6,1),(\g{su}_{1,2},1),(\g{e}^{-5}_7,1),(\g{sl}_3(\R),6),(\g{f}^4_4,1),(\g{g}^2_2,1),(\g{g}_2(\C),(1,1))    \}.\]
Notice that $\g{e}^{-5}_7$ cannot be contained in any other subalgebra in~$\widetilde{\mathcal{L}}(\g{g})$. However, by~\cite{BO1}, there is no non-semisimple maximal totally geodesic submanifold in~$M$, which implies that $(\g{e}^{-5}_7,1)$ induces a maximal totally geodesic submanifold. Now by Lemma~\ref{lemma:exceptions}, we have that $(\g{sl}_3(\R),6)$ and $(\g{g}_2(\C),(1,1))$ do not induce maximal totally geodesic submanifolds. Furthermore, by Table~\ref{table:tildeL(g)}, we have the following embeddings into subalgebras of $\g{g}$
\[\g{su}_{2,3}\subset\g{su}_5\oplus\g{su}_{2,3},\quad \g{e}^2_6\subset\g{su}_3\oplus\g{e}^2_6, \quad \g{su}_{1,2}\subset\g{su}_{1,2}\oplus\g{e}_6, \quad \g{f}^4_4\subset \g{f}^4_4\oplus\g{g}_2, \quad \g{g}^2_2\subset \g{f}_4\oplus\g{g}^2_{2}.       \]
Thus, by Lemma~\ref{lemma:nonmaxtg} \emph{iv)}, we have that $(\g{su}_{2,3},1)$, $(\g{e}^2_6,1)$, $(\g{su}_{1,2},1)$, $(\g{f}^4_4,1)$ and $(\g{g}^2_{2},1)$ do not induce maximal totally geodesic submanifolds. This concludes the proof.
\end{proof}

\section{Totally geodesic submanifolds in\\ exceptional symmetric spaces of type~IV}
\label{sect:totgeodcomplex}
In this section we will classify maximal totally geodesic submanifolds in exceptional symmetric spaces with complex isometry group. By duality this is equivalent to classifying maximal totally geodesic submanifolds in exceptional compact Lie groups.

Let $G/K$ be a  symmetric space
of compact type and $\sigma\in\Aut(G)$ be an involutive automorphism of~$G$ such that $\mathrm{Fix}^0(\sigma)\subset K\subset \mathrm{Fix}(\sigma)$, where $\mathrm{Fix}(\sigma)$ is the subset of~$G$ which is fixed by~$\sigma$ and $\mathrm{Fix}^0(\sigma)$ is its identity component. The Cartan embedding of~$G/K$ into~$G$ is the smooth map~$f$ given by
\[f\colon G/K \rightarrow G,\qquad gK \mapsto \sigma(g)g.\]

It was shown in \cite{Ikawa} that a totally geodesic submanifold in a compact Lie group is maximal if and only if it is a maximal subgroup or a Cartan embedding. However, in this section we will give an explicit list of all maximal totally geodesic submanifolds in exceptional symmetric spaces with complex isometry group.

Maximal semisimple regular subalgebras of~$\g{g}$ and maximal S-subalgebras can be found in~\cite{Dynkin} as mentioned above. Real forms are well known, see e.g.~\cite{helgason}. Therefore, by Lemma~\ref{lemma:maximalcomplexsemisimple}, one can obtain the set $\mathcal{L}(\g{g})$ for $\g{g}$ a realification of a exceptional simple complex Lie algebra, see Definition~\ref{def:Ltilde}; it is given in Table~\ref{table:L(g)complex}.

\begin{lemma}
	\label{lemma:nonsemi-complex}
Let $\g{g}$ be equal to $\g{e}_6(\C)$  or $\g{e}_7(\C)$ and let $\g{h}$ be one of the following two subalgebras of $\g{g}$:
	\[
		\g{so}_{10}(\C)  \subset \g{e}_6(\C), \qquad\;  \g{e}_6(\C) \subset \g{e}_7(\C).
\]
	Let $G$ be the simply connected Lie group with Lie algebra~$\g{g}$ and let $K$ be a maximal compact subgroup of~$G$. Assume the Lie algebra~$\g{h}$ is canonically embedded into~$\g{g}$ with respect to the Cartan decomposition~$\g{g} = \g{k} \oplus \g{p}$. Let $H$ be the connected closed subgroup of~$G$ with Lie algebra~$\g{h}$. Then a totally geodesic orbit of the $H$-action on the symmetric space~$M = G/K$ is not maximal.
\end{lemma}
\begin{proof}
By~\cite{Dynkin}, there is a reductive subalgebra isomorphic to $\g{h}\oplus\g{s}$, where $\g{s}$ is a $1$-dimensional complex subalgebra of~$\g{g}$, which is canonically embedded with respect to the Cartan decomposition of $\g{g}=\g{k}\oplus\g{p}$ by the same kind of argument as in the first paragraph of the proof of Lemma~\ref{lemma:nonsemi}.  Let $\Sigma$ be the totally geodesic orbit $H\cdot o$, where $H$ is the connected subgroup of~$G$ with Lie algebra $\g{h}$.
Moreover, $\theta$ leaves $\g{h}\oplus\g{s}$ invariant and maps $\g{s}$ onto itself, since it is the center of~$\g{h}\oplus\g{s}$ and $\theta$ is a Lie algebra automorphism. Thus $\g{s}$ is canonically embedded in~$\g{g}$ with respect to $\g{g}=\g{k}\oplus\g{p}$, implying that $\g{s}=(\g{k}\cap\g{s})\oplus(\g{p}\cap\g{s})$.

If $\dim_\R (\g{p}\cap \g{s})\neq0$, then $\Sigma$ is properly contained in a proper totally geodesic submanifold of $M$.
If $\dim_\R (\g{p}\cap \g{s})=0$, then $\g{s}$ has dimension two and the rank of $(\g{s}\oplus\g{h})\cap \g{k}$ is bigger than the rank of $\g{k}$, which leads to a contradiction in both cases. Indeed, if $\g{h}\simeq\g{so}_{10}(\C)$, then $\g{h}\cap \g{k}\simeq\g{so}_{10}$ and $\g{k}\simeq\g{e}_6$, and if $\g{h}\simeq\g{e}_6(\C)$, then $\g{h}\cap\g{k}\simeq\g{e}_6$ and $\g{k}\simeq\g{e}_7$.\qedhere
\end{proof}

\begin{table}[h]\caption{$\mathcal{L}(\g{g})$ for each exceptional simple complex Lie algebra~$\g{g}$. Notice that we indicate the different isometry classes for a given isomorphism class of a subalgebra by writing all their possible Dynkin indices separated by commas.}\label{table:L(g)complex}
	\begin{tabular}{cl}	
		\hline
		$\g{g}_2(\C)$ & 	$(\g{sl}_2(\C)\oplus\g{sl}_2(\C),(3,1)),(\g{sl}_2(\C),28), (\g{sl}_3(\C),1),(\g{g}_2^2,1)$
		\\ \hline
		\multirow{2}{*}{$\g{f}_4(\C)$} & 	$(\g{sl}_2(\C)\oplus\g{sp}_3(\C),(1,1)), (\g{sl}_3(\C)\oplus\g{sl}_3(\C),(1,2)),(\g{sl}_2(\C),156),(\g{so}_9(\C),1),$\\
		&$(\g{sl}_2(\C)\oplus\g{g}_2(\C),(8,1)),(\g{f}_4^{-20},1),(\g{f}_4^{4},1)$
		\\ \hline
		\multirow{3}{*}{$\g{e}_6(\C)$}	
		&
		$(\g{sl}_3(\C),9),(\g{g}_2(\C),3),(\g{sl}_3(\C)\oplus\g{g}_2(\C),(2,1)),(\g{sp}_4(\C),1),(\g{f}_4(\C),1),$	
		\\
		& $ (\g{so}_{10}(\C),1),(\g{sl}_3(\C)\oplus\g{sl}_3(\C)\oplus\g{sl}_3(\C),(1,1,1)),(\g{sl}_2(\C)\oplus\g{sl}_6(\C),(1,1)),$
		\\
		&$(\g{e}_6^6,1),(\g{e}_6^2,1),(\g{e}_6^{-26},1),(\g{e}_6^{-14},1)$
		\\ \hline
		\multirow{3}{*}{$\g{e}_7(\C)$}	
		&
		$(\g{sl}_2(\C),231,399),(\g{sl}_3(\C),21),(\g{sl}_2(\C)\oplus\g{sl}_2(\C),(15,24)),(\g{sl}_2(\C)\oplus\g{g}_2(\C),(7,2)), $	
		\\
		
		& $(\g{sp}_3(\C)\oplus \g{g}_2(\C),(1,1)),(\g{e}_6(\C),1), (\g{sl}_{2}(\C)\oplus\g{so}_{12}(\C),(1,1)),(\g{sl}_8(\C),1),$
		\\
		&  $(\g{sl}_{3}(\C)\oplus\g{sl}_6(\C),(1,1)),(\g{sl}_2(\C)\oplus\g{f}_4(\C),(3,1)),(\g{e}_7^{-5},1),(\g{e}_7^{7},1),(\g{e}_7^{-25},1)$
		\\ \hline
		\multirow{3}{*}{$\g{e}_8(\C)$}
		
		&
		$(\g{sl}_2(\C),520,760,1240),(\g{so}_5(\C),12),(\g{sl}_2(\C)\oplus\g{sl}_3(\C),(16,6)),$
		\\
		& $(\g{f}_4(\C)\oplus\g{g}_2(\C),(1,1)),(\g{so}_{16}(\C),1),(\g{sl}_9(\C),1),(\g{sl}_5(\C)\oplus\g{sl}_5(\C),(1,1)),$
		
		\\
		& $(\g{sl}_3(\C)\oplus\g{e}_6(\C),(1,1)),(\g{sl}_2(\C)\oplus\g{e}_7(\C),(1,1)),(\g{e}^8_8,1),(\g{e}_8^{-24},1)$\\ \hline
	\end{tabular}
	
\end{table}
\begin{theorem}
	\label{th:tablescomplex}
	Let $M=G/K$  be a symmetric space with $G$ an exceptional simple complex Lie group. Let $\Sigma$ be a maximal totally geodesic submanifold of $M$. Then $\Sigma$ is isometric to one of the spaces listed in Tables~\ref{table:g2}, \ref{table:f4}, \ref{table:e6}, \ref{table:e7}, \ref{table:e8}. Conversely, every space listed in these tables can be isometrically embedded as a maximal totally geodesic submanifold of~$M$.
\end{theorem}
\begin{proof}
Let $\Sigma$ be a maximal totally geodesic submanifold in~$M$. If $\Sigma$ is a non-semisimple totally geodesic submanifold, it is one of the examples found by Berndt and Olmos~\cite{BO1}, which we include in our tables.
Let us now assume that $\Sigma$ is semisimple. Then, by Corollary~\ref{cor:correspondence}, there is some Lie subalgebra $\g{g}_{\Sigma}$ of $\g{g}$ which is isometric to some subalgebra in~$\mathcal{L}(\g{g})$.
	
Let $M$ be equal to~$\mathsf{G}_2(\C)/\mathsf{G}_2$, $\mathsf{F}_2(\C)/\mathsf{F}_4$ or $\mathsf{E}_8(\C)/\mathsf{E}_8$. By~\cite[Corollary 4.3]{BO1}, we have that every maximal totally geodesic submanifold of $M$ is semisimple and therefore $\mathcal{L}(\g{g})$ is in one-to-one correspondence with the isometry classes of maximal totally geodesic submanifolds of $M$ by Corollary~\ref{cor:correspondence}.
	
	Now let $M=\mathsf{E}_{6}(\C)/\mathsf{E}_{6}$ and $\g{g}=\g{e}_6(\C)$.  Let $(\g{g}_{\Sigma},\ind_D(\g{g}_{\Sigma}))\in\mathcal{L}(\g{g})\setminus\{(\g{so}_{10}(\C),1)\}$. Let $\Sigma$ be the corresponding  maximal semisimple totally geodesic submanifold. If $\Sigma$ is not maximal, then it must be contained in $\R\times \mathsf{SO}_{10}(\C)/\mathsf{SO}_{10}$, which is the only non-semisimple maximal totally geodesic submanifold in~$M$ according to Berndt and Olmos~\cite{BO1}. However, if this is the case, since $\g{g}_{\Sigma}$ is semisimple, $\g{g}_{\Sigma}$ is contained in~$ \g{so}_{10}(\C)$, contradicting the fact that $\g{g}_{\Sigma}$ is a maximal non-compact type subalgebra in~$\g{g}$. Hence $\Sigma$ is a maximal totally geodesic submanifold. By Lemma~\ref{lemma:nonsemi-complex}, we have that $(\g{so}_{10}(\C),1)$ does not induce a maximal totally geodesic orbit.
	
	Finally, let $M=\mathsf{E}_{7}(\C)/\mathsf{E}_{7}$ and $\g{g}=\g{e}_7(\C)$.
	By a similar argument as above every subalgebra in~$\mathcal{L}(\g{g})\setminus\{(\g{e}_{6}(\C),1)\}$ induces a maximal totally geodesic submanifold. By Lemma~\ref{lemma:nonsemi-complex}, we have that $(\g{e}_{6}(\C),1)$ does not induce a maximal totally geodesic orbit.	
\end{proof}

\section{Proofs of the main theorems}\label{sect:pfs}

\begin{proof}[Proof of Theorem~\ref{mainth:a}]
Follows by combining Theorems~\ref{th:tablesreal} and~\ref{th:tablescomplex}.
\end{proof}

\begin{lemma}\label{lemma:DynCls}
The Dynkin index of the following subalgebras~$\g{h} \subset \g{g}$, $k \ge 1$, where the embeddings $\g{h} \hookrightarrow \g{g}$ are given by $A \mapsto \begin{psmallmatrix}
A &  \\ & 0 \end{psmallmatrix}$, is one:
\begin{inparaenum}[i)]
\item \label{subalgii} $\g{so}_n(\C) \subset \g{so}_{n+k}(\C),  n \ge 5$;
\item \label{subalgi} $\g{sl}_n(\C) \subset \g{sl}_{n+k}(\C), n \ge 2$;
\item \label{subalgiii} $\g{sp}_n(\C) \subset \g{sp}_{n+k}(\C),  n \ge 1$. Moreover, the Dynkin index of the subalgebra $\g{so}_4(\C) \subset \g{so}_{5}(\C)$ is~$(1,1)$.
\end{inparaenum}
\end{lemma}

\begin{proof}
Assume $\g{h}$ is a regular subalgebra of~$\g{g}$. Then a root system of~$\g{h}$ is a subset of a root system of~$\g{g}$ and we can apply the following simple observation. If $\g{h}$ contains a root space of~$\g{g}$ corresponding to a longest root with respect to some Cartan subalgebra~$\g{a}$ of~$\g{g}$, then the length of the longest root of~$\g{h}$ and the length of the longest root of~$\g{g}$ agree, and the Dynkin index of~$\g{h}$ in~$\g{g}$ is one.
This holds in particular if the roots systems of both~$\g{h}$ and~$\g{g}$ contain roots of different lengths, which shows that the subalgebras~\emph{\ref{subalgiii})} and $\g{so}_n(\C) \subset \g{so}_{n+2}(\C)$, for $n \ge 5$ odd, have Dynkin index one.
Using the multiplicative property of the Dynkin index and the chain of inclusions $\g{so}(n,\C) \subset \g{so}(n+1,\C) \subset \g{so}(n+2,\C)$, it now follows inductively that all the subalgebras~\emph{\ref{subalgii})} have Dynkin index one.
The above observation also applies if $\g{g}$ is of type $\mathsf{A}_n$ or $\mathsf{D}_n$, $n \ge 4$, and $\g{h}$ is a maximal semisimple regular subalgebra, since then all roots of the extended Dynkin diagram are of the same length, see~\cite[Ch.~1, \S3, Table~4]{Onisbook2}. Using induction and the multiplicative property of the Dynkin index this shows that the subalgebras~\emph{\ref{subalgi})} all have Dynkin index one.
\end{proof}

\begin{proof}[Proof of Theorem~\ref{th:index}]
Examples of totally geodesic submanifolds~$\Sigma$ in irreducible symmetric spaces~$M$ with $i(M)=\codim(\Sigma)$ are given in~\cite[Table~1]{BO5}.
For the exceptional symmetric spaces~$M$, these pairs~$(M,\Sigma)$ are the following and the theorem can be proved in these cases by looking up the Dynkin indices in our Tables~\ref{table:g2}, \ref{table:f4}, \ref{table:e6}, \ref{table:e7}, \ref{table:e8}.

\begin{itemize}
\item
$(\mathsf{G}_2/\mathsf{SO}_4,\mathsf{SL}_3(\R)/\mathsf{SO}_3)$,
$(\mathsf{G}_2(\C)/\mathsf{G}_2,\mathsf{G}_2/\mathsf{SO}_4)$,
$(\mathsf{G}_2(\C)/\mathsf{G}_2,\mathsf{SL}_3(\C)/\mathsf{SU}_3)$,
\item
$(\mathsf{F}^4_4/\mathsf{Sp}_3\mathsf{Sp}_1,\mathsf{SO}^0_{4,5}/\mathsf{SO}_{4}\times\mathsf{SO}_{5})$,
$(\mathsf{F}^{-20}_4/\mathsf{Spin}_9,\mathsf{SO}^0_{1,8}/\mathsf{SO}_{8})$,\\
$(\mathsf{F}^{-20}_4/\mathsf{Spin}_9,\mathsf{Sp}_{1,2}/\mathsf{Sp}_{1}\times\mathsf{Sp}_{2}),$
$(\mathsf{F}_4(\C)/\mathsf{F}_4,\mathsf{SO}_9(\C)/\mathsf{SO}_9),$
\item
$(\mathsf{E}^6_6/\mathsf{Sp}_4,\mathsf{F}^4_4/\mathsf{Sp}_3\mathsf{Sp}_1)$,
$(\mathsf{E}^2_6/\mathsf{SU}_6\mathsf{Sp}_1,\mathsf{F}^4_4/\mathsf{Sp}_3\mathsf{Sp}_1)$,
$(\mathsf{E}^{-14}_6/\mathsf{Spin}_{10}\mathsf{U}_1,\mathsf{SO}^*_{10}/\mathsf{U}_5)$,\\
$(\mathsf{E}^{-26}_6/\mathsf{F}_{4},\mathsf{F}^{-20}_4/\mathsf{Spin}_9)$,
$(\mathsf{E}_6(\C)/\mathsf{E}_{6},\mathsf{F}_4(\C)/\mathsf{F}_4)$,
\item
$(\mathsf{E}^{-5}_7/\mathsf{SO}_{12}\mathsf{Sp}_1,\mathsf{E}^2_6/\mathsf{SU}_6\mathsf{Sp}_1)$,
$(\mathsf{E}^{-25}_7/\mathsf{E}_{6}\mathsf{U}_1,\mathsf{E}^{-14}_6/\mathsf{Spin}_{10}\mathsf{U}_1)$,
$(\mathsf{E}^{7}_7/\mathsf{SU}_{8},\R\times\mathsf{E}^6_6/\mathsf{Sp}_4)$,\\
$(\mathsf{E}_7(\C)/\mathsf{E}_{7},\R\times\mathsf{E}_6(\C)/\mathsf{E}_{6})$,
\item
$(\mathsf{E}^{-24}_8/\mathsf{E}_{7}\mathsf{Sp}_1,\mathsf{E}^{-5}_7/\mathsf{SO}_{12}\mathsf{Sp}_1)$,
$(\mathsf{E}^{8}_8/\mathsf{SO}_{16},\mathsf{SL}_{2}(\R)/\mathsf{SO}_2\times\mathsf{E}^7_7/\mathsf{SU}_8)$,\\
$(\mathsf{E}_8(\C)/\mathsf{E}_{8},\mathsf{SL}_{2}(\C)/\mathsf{SU}_{2}\times \mathsf{E}_7(\C)/\mathsf{E}_7)$.
\end{itemize}

It remains to show the assertion of the theorem for the classical spaces.
We have to consider the pairs of spaces~$(M,\Sigma)$, where $i(M)=\codim(\Sigma)$, given in Table~\ref{table:ixtg}, where we have partially reproduced the contents of~\cite[Table~1]{BO5}.

\begin{table}[h!]
	\caption{Examples of totally geodesic submanifolds~$\Sigma$ in classical symmetric spaces~$M$ with $\codim(\Sigma)=i(M)$.}
	\label{table:ixtg}

\begin{tabular}{llll}
$M$ & $\Sigma$ & $i(M)$ & Conditions \\ \hline
$\mathsf{SU}_{r,r+k}/\mathsf{S}(\mathsf{U}_{r}\times\mathsf{U}_{r+k})$ & $\mathsf{SU}_{r,r+k-1}/\mathsf{S}(\mathsf{U}_{r}\times\mathsf{U}_{r+k-1})$ & $2r$ & $r \ge 1, k \ge 1$ \\
$\mathsf{SU}_{r,r}/\mathsf{S}(\mathsf{U}_{r}\times\mathsf{U}_{r})$ & $\mathsf{SU}_{r-1,r}/\mathsf{S}(\mathsf{U}_{r-1}\times\mathsf{U}_{r})$ & $2r$ & $r \ge 3$ \\ \hline
$\mathsf{SO}^0_{r,r+k}/\mathsf{SO}_{r}\times\mathsf{SO}_{r+k}$ & $\mathsf{SO}^0_{r,r+k-1}/\mathsf{SO}_{r}\times\mathsf{SO}_{r+k-1}$ & $r$ & $r \ge 1, k \ge 1$ \strt \\
$\mathsf{SO}^0_{r,r}/\mathsf{SO}_{r}\times\mathsf{SO}_{r}$ & $\mathsf{SO}^0_{r-1,r}/\mathsf{SO}_{r-1}\times\mathsf{SO}_{r}$ & $r$ & $r \ge 4$ \\ \hline
$\mathsf{Sp}_{2,2}/\mathsf{Sp}_{2}\times\mathsf{Sp}_{2}$ & $\mathsf{Sp}_2(\C)/\mathsf{Sp}_{2}$ & $6$ & \\
$\mathsf{Sp}_{r,r+k}/\mathsf{Sp}_{r}\times\mathsf{Sp}_{r+k}$ & $\mathsf{Sp}_{r,r+k-1}/\mathsf{Sp}_{r}\times\mathsf{Sp}_{r+k-1}$ & $4r$ & $r \ge 1, k \ge 1$ \\
$\mathsf{Sp}_{r,r}/\mathsf{Sp}_{r}\times\mathsf{Sp}_{r}$ & $\mathsf{Sp}_{r-1,r}/\mathsf{Sp}_{r-1}\mathsf{Sp}_{r}$ & $4r$ & $r \ge 3$ \\ \hline
$\mathsf{SL}_{r+1}(\R)/\mathsf{SO}_{r+1}$ & $\R \times \mathsf{SL}_{r}(\R)/\mathsf{SO}_{r}$ & $r$ & $r \ge 2$ \\ \hline
$\mathsf{SU}^*_{6}/\mathsf{Sp}_{3}$ & $\mathsf{SL}_{3}(\C)/\mathsf{SU}_{3}$ & $6$ & \\
$\mathsf{SU}^*_{8}/\mathsf{Sp}_{4}$ & $\mathsf{Sp}_{2,2}/\mathsf{Sp}_{2}\mathsf{Sp}_{2}$ & $11$ & \\
$\mathsf{SU}^*_{2r+2}/\mathsf{Sp}_{r+1}$ & $\R \times \mathsf{SU}^*_{2r}/\mathsf{Sp}_{r}$ & $4r$ & $r \ge 4$ \\ \hline
$\mathsf{Sp}_{r}(\R)/\mathsf{U}_{r}$ & $\mathsf{Sp}_{1}(\R)/\mathsf{U}_{1} \times \mathsf{Sp}_{r-1}(\R)/\mathsf{U}_{r-1}$ & $2r-2$ & $r \ge 3$ \\ \hline
$\mathsf{SO}^*_{4r}/\mathsf{U}_{2r}$ & $\mathsf{SO}^*_{4r-2}/\mathsf{U}_{2r-1}$ & $4r-2$ & $r \ge 3$ \\
$\mathsf{SO}^*_{4r+2}/\mathsf{U}_{2r+1}$ & $\mathsf{SO}^*_{4r}/\mathsf{U}_{2r}$ & $4r$ & $r \ge 2$ \\ \hline
$\mathsf{SL}_3(\C)/\mathsf{SU}_{3}$ & $\mathsf{SL}_3(\R)/\mathsf{SO}_{3}$ & $3$ & \\
$\mathsf{SL}_4(\C)/\mathsf{SU}_{4}$ & $\mathsf{Sp}_2(\C)/\mathsf{Sp}_{2}$ & $5$ & \\
$\mathsf{SL}_{r+1}(\C)/\mathsf{SU}_{r+1}$ & $\R \times \mathsf{SL}_{r}(\C)/\mathsf{SU}_{r}$ & $2r$ & $r \ge 4$ \\ \hline
$\mathsf{SO}_{2r+1}(\C)/\mathsf{SO}_{2r+1}$ & $\mathsf{SO}_{2r}(\C)/\mathsf{SO}_{2r}$ & $2r$ & $r \ge 2$ \\ \hline
$\mathsf{Sp}_{r}(\C)/\mathsf{Sp}_{r}$ & $\mathsf{Sp}_{1}(\C)/\mathsf{Sp}_{1} \times \mathsf{Sp}_{r-1}(\C)/\mathsf{Sp}_{r-1}$ & $4r-4$ & $r \ge 3$ \\ \hline
$\mathsf{SO}_{2r}(\C)/\mathsf{SO}_{2r}$ & $\mathsf{SO}_{2r-1}(\C)/\mathsf{SO}_{2r-1}$ & $2r-1$ & $r \ge 4$ \\
\end{tabular}
\end{table}

Among the pairs of spaces given in the table, there are some infinite series and also some isolated examples in low dimensions. For the infinite series, the assertion of the theorem follows in all cases from Lemma~\ref{lemma:DynCls} and the multiplicative property of the Dynkin index.
We will treat the remaining isolated examples individually. First note that the Dynkin index of~$\g{sp}_2(\C)$ in~$\g{sp}_{2,2}$ is~$(1,1)$, as follows from Lemma~\ref{lemma:DynCls} since the complexifications are $\g{sp}_2(\C) \oplus \g{sp}_2(\C)$ and~$\g{sp}_4(\C)$.
The Dynkin index of~$\g{sl}_3(\C)$ in~$\g{su}^*_6$ is $(1,1)$, as the complexifications are $\g{sl}_3(\C) \oplus \g{sl}_3(\C)$ and~$\g{sl}_6(\C)$. The subalgebra $\g{sl}_3(\R)$ is a real form of~$\g{sl}_3(\C)$ and therefore has Dynkin index~one. The subalgebra~$\g{sp}_2(\C)$ of~$\g{sl}_4(\C)$ has Dynkin index one by Lemma~\ref{lemma:DynCls}, since it corresponds to the subalgebra~$\g{so}_5(\C) \subset \g{so}_6(\C)$ under the isomorphism $\g{so}_6(\C) \simeq \g{sl}_4(\C)$. Finally, it remains to determine the Dynkin index of the subalgebra~$\g{sp}_{2,2}$ in~$\g{su}^*_8$. Since the Dynkin indices of $\g{sp}_2(\C) \subset \g{sp}_4(\C)$ and $\g{sl}_4(\C) \subset \g{sl}_8(\C)$ are one by~Lemma~\ref{lemma:DynCls}, and we know from the previous case that the Dynkin index of~$\g{sp}_2(\C) \subset \g{sl}_4(\C)$ is one, using the multiplicative property of the Dynkin index, we conclude from the following diagram
\[
\begin{array}{ccc}
  \g{sp}_2(\C) & \stackrel1{\longrightarrow} & \g{sl}_4(\C) \\
  {\scriptstyle 1}\downarrow &  & {\scriptstyle 1}\downarrow  \\
  \g{sp}_{4}(\C) & \longrightarrow & \g{sl}_8(\C)
\end{array}
\]
that the Dynkin index of $\g{sp}_{4}(\C)$ in~$\g{sl}_8(\C)$ is also one.
\end{proof}

\begin{table}[p]
	\caption[]{Maximal totally geodesic submanifolds of symmetric spaces of $\mathsf{G}_2$-type.}
	\label{table:g2}
	\normalsize
	\begin{tabular}{lllp{11ex}l}
		$M$ & $\Sigma$ & Dynkin index & Reflective? & $\dim \Sigma$
		\strt \\ \hline
		\multirow{2}{12ex}{$\mathsf{G}^2_2/\mathsf{SO}_4$}
		& $\mathsf{SL}_2(\R)/\mathsf{SO}_2\times\mathsf{SL}_2(\R)/\mathsf{SO}_2$ & $(3,1)$ & Yes & $4$
		\strt \\ \cline{2-5}
		& $\mathsf{SL}_3(\R)/\mathsf{SO}_3$ & $1$ & No & $5$
		\strt \\ \cline{2-5}
		& $\mathsf{SU}_{1,2}/\mathsf{S(U_1\times U_2)}$ & $1$ & No & $4$
		\strt \\ \cline{2-5}
		& $\mathsf{SL}_2(\R)/\mathsf{SO}_2$ & $28$ & No & $2$
		\strt \\ \hline
		\multirow{2}{12ex}{$\mathsf{G}_2(\C)/\mathsf{G}_2$}
		& $\mathsf{SL}_2(\C)/\mathsf{SU}_2\times\mathsf{SL}_2(\C)/\mathsf{SU}_2$ & $(3,1)$ & Yes & $6$
		\strt \\ \cline{2-5}
		& $\mathsf{SL}_2(\C)/\mathsf{SU}_2$ & $28$ & No & $3$
		\strt \\ \cline{2-5}
		& $\mathsf{SL}_{3}(\C)/\mathsf{SU}_3$ & $1$ & No & $8$
		\strt \\ \cline{2-5}
		& $\mathsf{G}^2_2/\mathsf{SO}_4$ & $1$ & Yes & $8$ \strt
	\end{tabular}
	\bigskip\bigskip
\end{table}

\begin{table}[p]
	\caption[]{Maximal totally geodesic submanifolds of symmetric spaces of $\mathsf{F}_4$-type.}
	\label{table:f4}
	\normalsize
	\begin{tabular}{lllp{11ex}l}
		$M$ & $\Sigma$ & Dynkin index & Reflective? & $\dim \Sigma$
		\strt \\ \hline
		\multirow{2}{12ex}{$\mathsf{F}^4_4/\mathsf{Sp}_3\mathsf{Sp}_1$}
		& $\mathsf{SL}_3(\R)/\mathsf{SO}_3\times\mathsf{SL}_3(\R)/\mathsf{SO}_3$ & $(1,2)$ & No & $10$
		\strt \\ \cline{2-5}
		& $\mathsf{SU}_{1,2}/\mathsf{S(U_1{\times} U_2)}{\times}\mathsf{SU}_{1,2}/\mathsf{S(U_1{\times} U_2)}$ & $(1,2)$ & No & $8$
		\strt \\ \cline{2-5}
		& $\mathsf{SL}_2(\R)/\mathsf{SO}_2\times\mathsf{Sp}_3(\R)/\mathsf{U}_3$ & $(1,1)$ & Yes & $14$
		\strt \\ \cline{2-5}
		& $\mathsf{Sp}_{1,2}/\mathsf{Sp}_1\times\mathsf{Sp}_2$ & $1$ & Yes & $8$
		\strt \\ \cline{2-5}
		& $\mathsf{SO}^0_{4,5}/\mathsf{SO}_4\times\mathsf{SO}_5$ & $1$ & Yes & $20$
		\strt \\ \cline{2-5}
		& $\mathsf{SL}_{2}(\R)/\mathsf{SO}_2$ & $156$ & No & $2$
		\strt \\ \cline{2-5}
		& $\mathsf{SL}_{2}(\R)/\mathsf{SO}_2\times\mathsf{G}^2_2/\mathsf{SO}_4$ & $(8,1)$ & No & $10$
		\strt \\ \hline
		\multirow{2}{12ex}{$\mathsf{F}^{-20}_4/\mathsf{Spin}_9$}
		& $\mathsf{SO}^0_{1,8}/\mathsf{SO}_8$ & $1$ & Yes  & $8$
		\strt \\ \cline{2-5}
		& $\mathsf{Sp}_{1,2}/\mathsf{Sp}_1\times \mathsf{Sp}_2$ & $1$ & Yes  & $8$
		\strt \\ \hline
		\multirow{5}{12ex}{$\mathsf{F}_4(\C)/\mathsf{F}_4$}
		& $\mathsf{SL}_{2}(\C)/\mathsf{SU}_{2}\times\mathsf{Sp}_{3}(\C)/\mathsf{Sp}_{3}$ & $(1,1)$ & Yes  & $24$
		\strt \\ \cline{2-5}
		& $\mathsf{SL}_{3}(\C)/\mathsf{SU}_{3}\times\mathsf{SL}_{3}(\C)/\mathsf{SU}_{3}$ & $(1,2)$ & No & $16$
		\strt \\ \cline{2-5}
		& $\mathsf{SL}_{2}(\C)/\mathsf{SU}_{2}$ & $156$ & No & $3$
		\strt \\ \cline{2-5}
		& $\mathsf{SO}_9(\C)/\mathsf{SO}_9$ & $1$ & Yes & $36$
		\strt \\ \cline{2-5}
		& $\mathsf{SL}_{2}(\C)/\mathsf{SU}_{2}\times\mathsf{G}_{2}(\C)/\mathsf{G}_{2}$ & $(8,1)$ & No & $17$
		\strt \\ \cline{2-5}
		& $\mathsf{F}^{4}_{4}/\mathsf{Sp}_{3}\mathsf{Sp}_1$ & $1$ & Yes & $28$
		\strt \\ \cline{2-5}
		& $\mathsf{F}^{-20}_{4}/\mathsf{Spin}_{9}$ & $1$ & Yes & $16$ \strt
		
	\end{tabular}
\end{table}

\begin{table}[p]
	\caption[]{Maximal totally geodesic submanifolds of symmetric spaces of $\mathsf{E}_6$-type.}
	\label{table:e6}
	\footnotesize
	\begin{tabular}{lllp{11ex}l}
		$M$ & $\Sigma$ & Dynkin index & Reflective? & $\dim \Sigma$
		\strt \\ \hline
		\multirow{2}{12ex}{$\mathsf{E}^6_6/\mathsf{Sp}_4$}
		& $(\mathsf{SL}_3(\R)/\mathsf{SO}_3)^3$ & $(1,1,1)$ & No & $15$
		\strt \\ \cline{2-5}
		& $\mathsf{SU}_{1,2}/\mathsf{S(U_1\times U_2)}\times\mathsf{SL}_{3}(\C)/\mathsf{SU}_3$ & $(1,1,1)$ & No & $12$
		\strt \\ \cline{2-5}
		& $\mathsf{SL}_2(\R)/\mathsf{SO}_2\times\mathsf{SL}_6(\R)/\mathsf{SO}_6$ & $(1,1)$ & Yes & $22$
		\strt \\ \cline{2-5}
		& $\mathsf{SU}^*_{6}/\mathsf{Sp}_3$ & $1$ & Yes & $14$
		\strt \\ \cline{2-5}
		& $\mathsf{SL}_3(\R)/\mathsf{SO}_3\times\mathsf{G}^2_2/\mathsf{SO}_4$ & $(2,1)$ & No & $13$
		\strt \\ \cline{2-5}
		& $\mathsf{Sp}_{2,2}/\mathsf{Sp}_2\times \mathsf{Sp}_2$ & $1$ & Yes & $16$
		\strt \\ \cline{2-5}
		& $\mathsf{Sp}_{4}(\R)/\mathsf{U}_4$ & $1$ & Yes & $20$
		\strt \\ \cline{2-5}
		& $\mathsf{F}^4_{4}/\mathsf{Sp}_3\mathsf{Sp}_1$ & $1$ & Yes & $28$
		\strt \\ \cline{2-5}
		& $\R\times\mathsf{SO}^0_{5,5}/\mathsf{SO}_5\times\mathsf{SO}_5$ & $1$ & Yes & $26$
		\strt \\ \hline
		\multirow{2}{12ex}{$\mathsf{E}^{2}_6/\mathsf{SU}_6\mathsf{Sp}_1$}
		& $\mathsf{SO}^0_{4,6}/\mathsf{SO}_4\times\mathsf{SO}_6$& $1$ & Yes & $24$
		\strt \\ \cline{2-5}
		& $\mathsf{SO}^*_{10}/\mathsf{U}_5$ & $1$ & Yes & $20$
		\strt \\ \cline{2-5}
		& $\mathsf{SL}_3(	\C)/\mathsf{SU}_3\times\mathsf{SL}_3(	\R)/\mathsf{SO}_3$ & $(1,1,1)$ & No & $13$
		\strt \\ \cline{2-5}
		& $(\mathsf{SU}_{1,2}/\mathsf{S(U_1\times U_2)})^3$ & $(1,1,1)$ & No & $12$
		\strt \\ \cline{2-5}
		& $\mathsf{SL}_{2}(\R)/\mathsf{SO}_2\times\mathsf{SU}_{3,3}/\mathsf{S(U_3\times U_3)}$ & $(1,1)$ & Yes & $20$
		\strt \\ \cline{2-5}
		& $\mathsf{SU}_{2,4}/\mathsf{S(U_2\times U_4)}$ & $1$ & Yes & $16$
		\strt \\ \cline{2-5}
		& $\mathsf{SU}_{1,2}/\mathsf{S(U_1\times U_2)}$ & $9$ & No & $4$
		\strt \\ \cline{2-5}
		& $\mathsf{SL}_{3}(\R)/\mathsf{SO}_3$ & $9$ & No & $5$
		\strt \\ \cline{2-5}
		& $\mathsf{G}^2_{2}/\mathsf{SO}_4$ & $3$ & No & $8$
		\strt \\ \cline{2-5}
		& $\mathsf{SU}_{1,2}/\mathsf{S(U_1\times U_2)}\times\mathsf{G}^2_{2}/\mathsf{SO}_4$ & $(2,1)$ & No & $12$
		\strt \\ \cline{2-5}
		& $\mathsf{Sp}_{1,3}/\mathsf{Sp}_1\times\mathsf{Sp}_3$ & $1$ & Yes & $12$
		\strt \\ \cline{2-5}
		& $\mathsf{Sp}_{4}(\R)/\mathsf{U}_4$ & $1$ & Yes & $20$
		\strt \\ \cline{2-5}
		& $\mathsf{F}^4_4/\mathsf{Sp}_3\mathsf{Sp}_1$ & $1$ & Yes & $28$
		\strt \\ \hline
		\multirow{2}{15ex}{$\mathsf{E}^{-14}_6/\mathsf{Spin}_{10}\mathsf{U}_1$}
		& $\mathsf{SO}^*_{10}/\mathsf{U}_5$ & $1$ & Yes & $20$
		\strt \\ \cline{2-5}
		& $\mathsf{SO}^0_{2,8}/\mathsf{SO}_2\times\mathsf{SO}_8$ & $1$ & Yes & $16$
		\strt \\ \cline{2-5}
		& $\mathsf{SL}_2(\R)/\mathsf{SO}_2\times\mathsf{SU}_{1,5}/\mathsf{S(U_1\times U_5)}$ & $(1,1)$ & Yes & $12$
		\strt \\ \cline{2-5}
		& $\mathsf{Sp}_{2,2}/\mathsf{Sp}_2\times\mathsf{Sp}_2$ & $1$ & Yes & $16$
		\strt \\ \cline{2-5}
		& $\mathsf{F}^{-20}_4/\mathsf{Spin}_9$ & $1$ & Yes & $16$
		\strt \\ \cline{2-5}
		& $\mathsf{SU}_{2,4}/\mathsf{S(U_2\times U_4)}$ & $1$ & Yes & $16$
		\strt \\ \hline
		\multirow{5}{12ex}{$\mathsf{E}^{-26}_6/\mathsf{F}_4$}
		&
		$\R\times\mathsf{SO}^0_{1,9}/\mathsf{SO}_9$ & $1$ & Yes  & $10$
		\strt \\ \cline{2-5}
		& $\mathsf{SU}^*_6/\mathsf{Sp}_3$ & $1$ & Yes & $14$
		\strt \\ \cline{2-5}
		&  $\mathsf{Sp}_{1,3}/\mathsf{Sp}_1\times\mathsf{Sp}_3$ & $1$ & Yes & $12$
		\strt \\ \cline{2-5}
		&  $\mathsf{F}^{-20}_4/\mathsf{Spin}_9$ & $1$ & Yes & $16$
		\strt \\
		\hline
		\multirow{5}{12ex}{$\mathsf{E}_6(\C)/\mathsf{E}_6$}
		& $\mathsf{SL}_{3}(\C)/\mathsf{SU}_{3}$ & $9$ & No  & $8$
		\strt \\ \cline{2-5}
		& $\mathsf{G}_{2}(\C)/\mathsf{G}_{2}$ & $3$ & No & $14$
		\strt \\ \cline{2-5}
		& $\mathsf{SL}_{3}(\C)/\mathsf{SU}_{3}\times \mathsf{G}_2(\C)/\mathsf{G}_2$ & $(2,1)$ & No & $22$
		\strt \\ \cline{2-5}
		& $\mathsf{Sp}_4(\C)/\mathsf{Sp}_4$ & $1$ & Yes & $36$
		\strt \\ \cline{2-5}
		& $\mathsf{F}_{4}(\C)/\mathsf{F}_{4}$ & $1$ & Yes & $52$
		\strt \\ \cline{2-5}
		& $\R\times\mathsf{SO}_{10}(\C)/\mathsf{SO}_{10}$ & $1$ & Yes & $46$
		\strt \\ \cline{2-5}
		& $(\mathsf{SL}_{3}(\C)/\mathsf{SU}_{3})^3$ & $(1,1,1)$ & No & $24$
		\strt \\ \cline{2-5}
		& $\mathsf{SL}_{2}(\C)/\mathsf{SU}_{2}\times \mathsf{SL}_6(\C)/\mathsf{SU}_6$ & $(1,1)$ & Yes & $38$
		\strt \\ \cline{2-5}
		& $\mathsf{E}^6_6/\mathsf{Sp}_4$ & $1$ & Yes & $42$
		\strt \\ \cline{2-5}
		& $\mathsf{E}^2_6/\mathsf{SU}_6\mathsf{Sp}_1$ & $1$ & Yes & $40$
		\strt \\ \cline{2-5}
		& $\mathsf{E}^{-26}_6/\mathsf{F}_4$ & $1$ & Yes & $26$
		\strt \\ \cline{2-5}
		& $\mathsf{E}^{-14}_6/\mathsf{Spin}_{10}\mathsf{U}_1$ & $1$ & Yes & $32$ \strt
	\end{tabular}
\end{table}

\begin{table}[p]
	\caption[]{Maximal totally geodesic submanifolds of symmetric spaces of $\mathsf{E}_7$-type.}
	\label{table:e7}
	\footnotesize
	\begin{tabular}{lllp{11ex}l}
		$M$ & $\Sigma$ & Dynkin index & Reflective? & $\dim \Sigma$
		\strt \\ \hline
		\multirow{2}{15ex}{$\mathsf{E}^{-5}_7/\mathsf{SO}_{12}\mathsf{Sp}_1$}
		& $\mathsf{E}^2_6/\mathsf{SU}_6\mathsf{Sp}_1$ & $1$ & Yes & $40$
		\strt \\ \cline{2-5}
		& $\mathsf{E}^{-14}_6/\mathsf{Spin}_{10}\mathsf{U}_1$ & $1$ & Yes & $32$
		\strt \\ \cline{2-5}
		& $\mathsf{SL}_2(\R)/\mathsf{SO}_2\times\mathsf{SO}^*_{12}/\mathsf{U}_6$ & $(1,1)$ & Yes & $32$
		\strt \\ \cline{2-5}
		& $\mathsf{SO}^0_{4,8}/\mathsf{SO}_4\times\mathsf{SO}_8$ & $1$ & Yes & $32$
		\strt \\ \cline{2-5}
		& $\mathsf{SU}_{4,4}/\mathsf{S(U_4\times U_4)}$ & $1$ & Yes & $32$
		\strt \\ \cline{2-5}
		& $\mathsf{SU}_{2,6}/\mathsf{S(U_2\times U_6)}$ & $1$ & Yes & $24$
		\strt \\ \cline{2-5}
		& $\mathsf{SL}_{3}(\R)/\mathsf{SO}_3\times\mathsf{SU}^*_6/\mathsf{Sp}_3$ & $(1,1)$ & No & $19$
		\strt \\ \cline{2-5}
		& $\mathsf{SU}_{1,2}/\mathsf{S(U_1\times U_2)}\times\mathsf{SU}_{2,4}/\mathsf{S(U_2\times U_4)}$ & $(1,1)$ & No & $20$
		\strt \\ \cline{2-5}
		& $\mathsf{SU}_{1,2}/\mathsf{S(U_1\times U_2)}$ & $21$ & No & $4$
		\strt \\ \cline{2-5}
		& $\mathsf{Sp}_{1,2}/\mathsf{Sp}_1\times \mathsf{Sp}_2\times\mathsf{G}^2_{2}/\mathsf{SO}_4$ & $(1,1)$ & No & $16$
		\strt \\ \hline
		\multirow{2}{12.67ex}{$\mathsf{E}^{-25}_7/\mathsf{E}_{6}\mathsf{U}_1$}
		& $\R\times\mathsf{E}^{-26}_6/\mathsf{F}_4$ & $1$ & Yes & $27$
		\strt \\ \cline{2-5}
		& $\mathsf{E}^{-14}_6/\mathsf{Spin}_{10}\mathsf{U}_1$ & $1$ & Yes & $32$
		\strt \\ \cline{2-5}
		& $\mathsf{SU}^*_8/\mathsf{Sp}_4$ & $1$ & Yes & $27$
		\strt \\ \cline{2-5}
		& $\mathsf{SU}_{2,6}/\mathsf{S(U_2\times U_6)}$ & $1$ & Yes & $24$
		\strt \\ \cline{2-5}
		& $\mathsf{SU}_{1,2}/\mathsf{S(U_1\times U_2)}\times\mathsf{SU}_{1,5}/\mathsf{S(U_1\times U_5)}$ & $(1,1)$ & No & $14$
		\strt \\ \cline{2-5}
		& $\mathsf{SL}_{2}(\R)/\mathsf{SO}_2\times \mathsf{F}^{-20}_4/\mathsf{Spin}_9$ & $(3,1)$ & No & $18$ \strt \\ \cline{2-5}
		& $\mathsf{SL}_2(\R)/\mathsf{SO}_2\times\mathsf{SO}^0_{2,10}/\mathsf{SO}_2\times\mathsf{SO}_{10}$ & $(1,1)$ & Yes & $22$
		\strt \\ \cline{2-5}
		& $\mathsf{SO}^*_{12}/\mathsf{U}_6$ & $1$ & Yes & $30$
		\strt \\ \hline
		\multirow{2}{12.67ex}{$\mathsf{E}^{7}_7/\mathsf{SU}_{8}$}
		& $\R\times\mathsf{E}^6_6/\mathsf{Sp}_4$ & $1$ & Yes & $43$
		\strt \\ \cline{2-5}
		& $\mathsf{E}^2_6/\mathsf{SU}_6\mathsf{Sp}_1$ & $1$ & Yes & $40$
		\strt \\ \cline{2-5}
		& $\mathsf{SL}_2(\R)/\mathsf{SO}_{2}\times\mathsf{SO}^0_{6,6}/\mathsf{SO}_6\times\mathsf{SO}_6$ & $(1,1)$ & Yes & $38$
		\strt \\ \cline{2-5}
		& $\mathsf{SO}^*_{12}/\mathsf{U}_6$ & $1$ & Yes & $30$
		\strt \\ \cline{2-5}
		& $\mathsf{SL}_8(\R)/\mathsf{SO}_8$ & $1$ & Yes & $35$
		\strt \\ \cline{2-5}
		& $\mathsf{SU}_{4,4}/\mathsf{S(U_4\times U_4)}$ & $1$ & Yes & $32$
		\strt \\ \cline{2-5}
		& $\mathsf{SU}^*_8/\mathsf{Sp}_4$ & $1$ & Yes & $27$
		\strt \\ \cline{2-5}
		& $\mathsf{SL}_{3}(\R)/\mathsf{SO}_3\times\mathsf{SL}_6(\R)/\mathsf{SO}_6$ & $(1,1)$ & No & $25$
		\strt \\ \cline{2-5}
		& $\mathsf{SU}_{1,2}/\mathsf{S(U_1\times U_2)}\times\mathsf{SU}_{3,3}/\mathsf{S(U_3\times U_3)}$ & $(1,1)$ & No & $22$
		\strt \\ \cline{2-5}
		& $\mathsf{SL}_{2}(\R)/\mathsf{SO}_2$ & $231,399$ & No & $2$
		\strt \\ \cline{2-5}
		& $\mathsf{SL}_{3}(\R)/\mathsf{SO}_3$ & $21$ & No & $5$
		\strt \\ \cline{2-5}
		& $\mathsf{SL}_{2}(\R)/\mathsf{SO}_2\times\mathsf{SL}_{2}(\R)/\mathsf{SO}_2$ & $(15,24)$ & No & $4$
		\strt \\ \cline{2-5}
		& $\mathsf{SL}_{2}(\R)/\mathsf{SO}_2\times \mathsf{G}^2_2/\mathsf{SO}_4$ & $(7,2)$ & No & $10$
		\strt \\ \cline{2-5}
		& $\mathsf{Sp}_{3}(\R)/\mathsf{U}_3\times\mathsf{G}^2_2/\mathsf{SO}_4$ & $(1,1)$ & No & $20$
		\strt \\ \cline{2-5}
		& $\mathsf{SL}_{2}(\R)/\mathsf{SO}_2\times\mathsf{F}^4_4/\mathsf{Sp}_3\mathsf{Sp}_1$ & $(3,1)$ & No & $30$ \strt \\
		\hline
		\multirow{5}{12ex}{$\mathsf{E}_7(\C)/\mathsf{E}_7$}
		& $\mathsf{SL}_{2}(\C)/\mathsf{SU}_{2}$ & $231,399$ & No  & $3$
		\strt \\ \cline{2-5}
		& $\mathsf{SL}_{3}(\C)/\mathsf{SU}_{3}$ & $21$ & No  & $8$
		\strt \\ \cline{2-5}
		& $\mathsf{SL}_{2}(\C)/\mathsf{SU}_{2}\times \mathsf{SL}_{2}(\C)/\mathsf{SU}_{2}$ & $(15,24)$ & No & $6$
		\strt \\ \cline{2-5}
		& $\mathsf{SL}_{2}(\C)/\mathsf{SU}_{2}\times \mathsf{G}_2(\C)/\mathsf{G}_2$ & $(7,2)$ & No & $17$
		\strt \\ \cline{2-5}
		& $\mathsf{Sp}_3(\C)/\mathsf{Sp}_3\times \mathsf{G}_2(\C)/\mathsf{G}_2$ & $(1,1)$ & No & $35$
		\strt \\ \cline{2-5}
		& $\R\times\mathsf{E}_{6}(\C)/\mathsf{E}_{6}$ & $1$ & Yes & $79$
		\strt \\ \cline{2-5}
		& $\mathsf{SL}_{2}(\C)/\mathsf{SU}_{2}\times \mathsf{SO}_{12}(\C)/\mathsf{SO}_{12}$ & $(1,1)$ & Yes & $69$
		\strt \\ \cline{2-5}
		& $\mathsf{SL}_{8}(\C)/\mathsf{SU}_{8}$ & $1$ & Yes & $63$
		\strt \\ \cline{2-5}
		& $\mathsf{SL}_{3}(\C)/\mathsf{SU}_{3}\times \mathsf{SL}_6(\C)/\mathsf{SU}_6$ & $(1,1)$ & No & $43$
		\strt \\ \cline{2-5}
		& $\mathsf{SL}_{2}(\C)/\mathsf{SU}_{2}\times \mathsf{F}_4(\C)/\mathsf{F}_4$ & $(3,1)$ & No & $55$
		\strt \\ \cline{2-5}
		& $\mathsf{E}^{-5}_7/\mathsf{SO}_{12}\mathsf{Sp}_1$ & $1$ & Yes & $64$
		\strt \\ \cline{2-5}
		& $\mathsf{E}^7_7/\mathsf{SU}_8$ & $1$ & Yes & $70$
		\strt \\ \cline{2-5}
		& $\mathsf{E}^{-25}_7/\mathsf{E}_6\mathsf{U}_1$ & $1$ & Yes & $54$ \strt
	\end{tabular}

\end{table}

\begin{table}[p]
	\caption[]{Maximal totally geodesic submanifolds of symmetric spaces of $\mathsf{E}_8$-type.}
	\label{table:e8}
	\footnotesize
	\begin{tabular}{lllp{11ex}l}
		$M$ & $\Sigma$ & Dynkin index & Reflective? & $\dim \Sigma$
		\strt \\ \hline
		\multirow{2}{10.67ex}{$\mathsf{E}^{8}_8/\mathsf{SO}_{16}$}
		& $(\mathsf{SU}_{1,4}/\mathsf{S(U_1\times U_4)})^2$ & $(1,1)$ & No & $16$
		\strt \\ \cline{2-5}
		& $(\mathsf{SU}_{2,3}/\mathsf{S(U_2\times U_3)})^2$ & $(1,1)$ & No & $24$
		\strt \\ \cline{2-5}
		& $\mathsf{SL}_5(\R)/\mathsf{SO}_5\times\mathsf{SL}_5(\R)/\mathsf{SO}_5$ & $(1,1)$ & No & $28$
		\strt \\ \cline{2-5}
		& $\mathsf{SO}^0_{8,8}/\mathsf{SO}_8\times\mathsf{SO}_8$ & $1$ & Yes & $64$
		\strt \\ \cline{2-5}
		& $\mathsf{SO}^*_{16}/\mathsf{U}_8$ & $1$ & Yes & $56$
		\strt \\ \cline{2-5}
		& $\mathsf{SU}_{1,8}/\mathsf{S(U_1\times U_8)}$ & $1$ & No & $16$
		\strt \\ \cline{2-5}
		& $\mathsf{SU}_{4,5}/\mathsf{S(U_4\times U_5)}$ & $1$ & No & $40$
		\strt \\ \cline{2-5}
		& $\mathsf{SL}_{9}(\R)/\mathsf{SO}_9$ & $1$ & No & $44$
		\strt \\ \cline{2-5}
		& $\mathsf{SL}_{3}(\R)/\mathsf{SO}_3\times \mathsf{E}^6_6/\mathsf{Sp}_4$ & $(1,1)$ & No & $47$
		\strt \\ \cline{2-5}
		& $\mathsf{SU}_{1,2}/\mathsf{S(U_1\times U_2)}\times\mathsf{E}^2_{6}/\mathsf{SU}_6\mathsf{Sp}_1$ & $(1,1)$ & No & $44$
		\strt \\ \cline{2-5}
		& $\mathsf{E}^{-5}_{7}/\mathsf{SO}_{12}\mathsf{Sp}_1$ & $1$ & Yes & $64$
		\strt \\ \cline{2-5}
		& $\mathsf{SL}_{2}(\R)/\mathsf{SO}_2\times\mathsf{E}^7_7/\mathsf{SU}_8$ & $(1,1)$ & Yes & $72$
		\strt \\ \cline{2-5}
		& $\mathsf{SO}^0_{2,3}/\mathsf{SO}_2\times\mathsf{SO}_3$ & $12$ & No & $6$
		\strt \\ \cline{2-5}
		&
		$\mathsf{SO}^0_{1,4}/\mathsf{SO}_4$ & $12$ & No & $4$
		\strt \\ \cline{2-5}
		&
		 $\mathsf{SL}_{2}(\R)/\mathsf{SO}_2$ & $520, 760, 1240$ & No & $2$
		\strt \\ \cline{2-5}
		& $\mathsf{SL}_2(\R)/\mathsf{SO}_2\times\mathsf{SU}_{1,2}/\mathsf{S(U_1\times U_2)}$ & $(16,6)$ & No & $6$
		\strt \\ \cline{2-5}
		&
	$\mathsf{SL}_2(\R)/\mathsf{SO}_2\times\mathsf{SL}_{3}(\R)/\mathsf{SO}_3$ & $(16,6)$ & No & $7$
		\strt \\ \cline{2-5}
		&
		 $\mathsf{F}^4_4/\mathsf{Sp}_3\mathsf{Sp}_1\times\mathsf{G}^2_2/\mathsf{SO}_4$ & $(1,1)$ & No & $36$
		\strt \\ \cline{2-5}
		& $\mathsf{G}_2(\C)/\mathsf{G}_2\times\mathsf{SL}_2(\R)/\mathsf{SO}_2$ & $(1,1,8)$ & No & $16$  \strt \\
		\hline
			\multirow{2}{15ex}{$\mathsf{E}^{-24}_8/\mathsf{E}_{7}\mathsf{Sp}_1$}
		& $\mathsf{SU}_{1,4}/\mathsf{S(U_1\times U_4)}\times\mathsf{SU}_{2,3}/\mathsf{S(U_2\times U_3)}$ & $(1,1)$ & No & $20$
			\strt \\ \cline{2-5}
		& $\mathsf{SO}^0_{4,12}/\mathsf{SO}_4\times\mathsf{SO}_{12}$ & $1$ & Yes & $48$
		\strt \\ \cline{2-5}
		& $\mathsf{SO}^*_{16}/\mathsf{U}_8$ & $1$ & Yes & $56$
		\strt \\ \cline{2-5}
		& $\mathsf{SU}_{3,6}/\mathsf{S(U_3\times U_6)}$ & $1$ & No & $36$
		\strt \\ \cline{2-5}
		& $\mathsf{SU}_{2,7}/\mathsf{S(U_2\times U_7)}$ & $1$ & No & $28$
		\strt \\ \cline{2-5}
		& $\mathsf{SU}_{1,2}/\mathsf{S(U_1\times U_2)}\times \mathsf{E}^{-14}_6/\mathsf{Spin}_{10}\mathsf{U}_1$ & $(1,1)$ & No & $36$
		\strt \\ \cline{2-5}
		& $\mathsf{SL}_{3}(\R)/\mathsf{SO}_3\times\mathsf{E}^{-26}_{6}/\mathsf{F_4}$ & $(1,1)$ & No & $31$
		\strt \\ \cline{2-5}
		& $\mathsf{E}^{-5}_{7}/\mathsf{SO}_{12}\mathsf{Sp}_1$ & $1$ & Yes & $64$
		\strt \\ \cline{2-5}
		& $\mathsf{E}^{-25}_{7}/\mathsf{E}_6\mathsf{U}_1\times\mathsf{SL}_{2}(\R)/\mathsf{SO}_2$ & $(1,1)$ & Yes & $56$
		\strt \\ \cline{2-5}
		& $\mathsf{F}^{-20}_4/\mathsf{Spin}_{9}\times\mathsf{G}^2_2/\mathsf{SO}_4$ & $(1,1)$ & No & $24$
		\strt \\ \hline
		\multirow{5}{12ex}{$\mathsf{E}_8(\C)/\mathsf{E}_8$}
		& $\mathsf{SL}_{2}(\C)/\mathsf{SU}_{2}$ & $520,760,1240$ & No  & $3$
		\strt \\ \cline{2-5}
		& $\mathsf{SO}_{5}(\C)/\mathsf{SO}_{5}$ & $12$ & No  & $10$
		\strt \\ \cline{2-5}
		& $\mathsf{SL}_{2}(\C)/\mathsf{SU}_{2}\times \mathsf{SL}_{3}(\C)/\mathsf{SU}_{3}$ & $(16,6)$ & No & $11$
		\strt \\ \cline{2-5}
		& $\mathsf{F}_{4}(\C)/\mathsf{F}_{4}\times \mathsf{G}_2(\C)/\mathsf{G}_2$ & $(1,1)$ & No & $66$
		\strt \\ \cline{2-5}
		& $\mathsf{SO}_{16}(\C)/\mathsf{SO}_{16}$ & $1$ & Yes & $120$
		\strt \\ \cline{2-5}
		& $\mathsf{SL}_{9}(\C)/\mathsf{SU}_{9}$ & $1$ & No & $80$
		\strt \\ \cline{2-5}
		& $\mathsf{SL}_{5}(\C)/\mathsf{SU}_{5}\times \mathsf{SL}_{5}(\C)/\mathsf{SU}_{5}$ & $(1,1)$ & No & $48$
		\strt \\ \cline{2-5}
		& $\mathsf{SL}_{3}(\C)/\mathsf{SU}_{3}\times\mathsf{E}_6(\C)/\mathsf{E}_6$ & $(1,1)$ & No & $86$
		\strt \\ \cline{2-5}
		& $\mathsf{SL}_{2}(\C)/\mathsf{SU}_{2}\times \mathsf{E}_7(\C)/\mathsf{E}_7$ & $(1,1)$ & Yes & $136$
		\strt \\ \cline{2-5}
		& $\mathsf{E}^8_{8}/\mathsf{SO}_{16}$ & $1$ & Yes & $128$
		\strt \\ \cline{2-5}
		& $\mathsf{E}^8_8/\mathsf{E}_7\mathsf{Sp}_1$ & $1$ & Yes & $112$ \strt
	\end{tabular}
\end{table}

\end{document}